\documentclass[a4paper, 11pt]{amsart}

\usepackage{latexsym,amssymb}
\usepackage[utf8]{inputenc}
\usepackage[T1]{fontenc}

\usepackage{pdfsync}

\usepackage[dvips]{graphicx}

\usepackage{amsfonts,amsmath,amssymb,amsthm,color,amsbsy}
\usepackage{color}
\usepackage{url}
\usepackage{eucal}
\usepackage{enumerate}
\usepackage{graphicx}
\usepackage{xfrac}
\usepackage{url}
\usepackage{mathrsfs}
\usepackage{bm}
\usepackage[colorlinks=true, pdfstartview=FitV, linkcolor=blue,citecolor=blue, urlcolor=blue]{hyperref}

\newtheorem{theorem}{Theorem}[section]
\newtheorem{lemma}[theorem]{Lemma}
\newtheorem{proposition}[theorem]{Proposition}
\newtheorem{corollary}[theorem]{Corollary}
\newtheorem{remark}[theorem]{Remark}

\renewcommand{\geq}{\geqslant}
\renewcommand{\leq}{\leqslant}
\renewcommand{\tilde}{\widetilde}

\DeclareMathOperator*{\essinf}{ess\,inf\,}

\def\R{\mathbb{R}}
\def\C{\mathbb{C}}

\def\d{\mathrm{\,d}}
\def\part{\partial}

\def\Ker{\mathrm{Ker}\,}

\def\d{\mathrm{d}}

\def\ve{\varepsilon}

\def\part{\partial}

\def\R{\mathbb{R}}

\def\mm{\kern +2pt \raisebox{+0.5 pt}{$\shortmid$}\kern -2pt\hbox{$\multimap$}\kern +2pt}

   \newcommand{\be}{\begin{equation}}
  \newcommand{\ee}{\end{equation}}
\numberwithin{equation}{section}

\author{Aleksander \'{C}wiszewski}
\address{Faculty of Mathematics and Computer Science \newline Nicolaus Copernicus University \newline 87-100 Toru\'n, Poland}
\email{aleks@mat.umk.pl}

\author{Piotr Kokocki}
\address{Faculty of Mathematics and Computer Science \newline Nicolaus Copernicus University \newline 87-100 Toru\'n, Poland}
\email{pkokocki@mat.umk.pl}

\title[Standing Waves for...]{Standing Waves for Schr\"{o}dinger Equations with Kato-Rellich potentials}

\begin{document}

\maketitle
\vspace{-7mm}
\begin{abstract}
We show the existence of standing waves for the nonlinear Schr\"{o}dinger equation with Kato-Rellich type potential. We consider both resonant with the nonlinearity satisfying one of Landesman-Lazer type or sign conditions and non-resonant case where the linearization at infinity has zero kernel. The approach relies on the geometric and topological analysis of the parabolic semiflow associated to the involved elliptic problem. Tail estimates techniques and spectral theory of unbounded linear operators are used to exploit subtle compactness properties necessary for use of the Conley index theory due to Rybakowski. The obtained results extend those by Prizzi \cite{Prizzi-FM} for the non-resonant case, resolve the existence problem at resonance and complete those from \cite{Cw-Kr-2019} where the bifurcation of stationary solutions from infinity was studied.
\end{abstract}

\section{Introduction}

In the time-dependent nonlinear Schr\"{o}dinger type equation
$$
i \psi_t = -\Delta \psi+V(x)\psi + g(x,|\psi|)\frac{\psi}{|\psi|}, \quad x\in \R^N, \ t>0, 
$$
where $V:\R^{N}\to\R$ is a potential and $g:\R^N\times \R\to \C$ is a nonlinear map, searching for a standing wave solution having the form $\psi(x,t) = e^{-i \lambda t} u(x)$ with $\lambda\in\R$, leads to the following elliptic problem
$$
-\Delta u + V(x)u-\lambda u = f(x, u(x)), \quad x\in\R^N,  \leqno{(E)_\lambda}
$$ 
where $f:\R^N\times \R \to \R$ is a nonlinear map.\\
\indent Many authors have extensively studied the equation $(E)_\lambda$ by the use of various methods.
In \cite{Prizzi-FM} the existence of solutions by Conley index was obtained in the non-resonant case with Kato-Rellich type $V$. Bifurcation theory involving topological degree methods was used in \cite{Genoud} for bounded $V$ in order to obtain the existence of large solutions near eigenvalues, i.e. the so-called bifurcation from infinity. For the same bifurcation problems with bounded potential $V$ variational methods were applied in \cite{Kr-Sz}, whereas in \cite{Stuart1}, an argument relying on topological degree were employed. 
Conley index techniques were used in \cite{Cw-Kr-2019} to study bifurcations from infinity for $(E)_\lambda$ with potentials of Kato-Rellich type, while in \cite{MR4293056}, the same approach were utilized in the case of bounded potentials $V$. \\
\indent Throughout the paper we assume that $V$ is a Kato-Rellich type potential, that is, 
$$\leqno{(K\!R)}\qquad \qquad
\left\{\begin{gathered}
V=V_\infty+V_0, \ \text{where} \ V_\infty \in L^\infty(\R^N) \ \text{and} \ V_0\in L^p(\R^N) \ \text{with $p$ satisfying}\\
p\geq 2 \ \text{ if } \ 1\leq N\leq 3 \ \ \text{ and } \ \ p>N/2  
 \ \text{ if } \ 
N\geq 4.
\end{gathered}\right.$$
We also assume that the map $f:\R^N\times \R\to\R$ satisfies the conditions:\\[2pt]
\noindent\makebox[9mm][r]{$(f1)$} \parbox[t][][t]{165mm}{the map $f(\,\cdot\,,u)$ is Lebesgue measurable, for all $u\in\R$, and $c(x):=f(x,0)$ is an element of $L^{2}(\R^{N})$,}\\[2pt]
\noindent\makebox[9mm][r]{$(f2)$} \parbox[t][][t]{160mm}{there is a function $l:\R^N\to\R$ such that 
\begin{equation*}
|f(x,u)-f(x,v)|\leq l(x)|u-v|\quad \text{for a.e. } x\in\R^N \text{ and all } u,v\in\R,
\end{equation*}
and the following (stronger Kato-Rellich type) conditions are satisfied
$$\leqno{(PT)}\qquad \qquad
\left\{\begin{gathered}
l=l_\infty+l_0, \ \text{where} \ l_\infty \in L^\infty(\R^N) \ \text{and} \ l_0\in L^r(\R^N) \ \text{with $r$ satisfying}\\
r\geq 2  \ \   \text{if} \  \  N=1,  \ \ r>2 \ \ \text{if}\ \ N=2\ \ \text{and}\ \ r\geq N \  \text{if }  \ N\geq 3.
\end{gathered}\right.$$}\\
In some cases we shall also assume that\\[2pt]
\noindent\makebox[9mm][r]{$(f3)$} \parbox[t][][t]{165mm}{there exists $a:\R^N\to\R$ satisfying the condition $(KR)$ such that
$$
u\cdot f(x,u) \leq a(x)|u|^2 \quad \text{for a.e. } x\in\R^N \text{ and all } u\in\R.
$$}\\[2pt]
\indent Our study of the existence of standing waves and the related problem $(E)_\lambda$  strongly relies on the fact that $V$ is a Kato-Rellich type potential and its consequences for the Schr\"{o}dinger operator $-\Delta+V$, including its spectral properties. In this paper we shall consider the realization of $-\Delta+V$ in $L^2(\R^N)$, which appears to be a closed operator in $L^2(\R^N)$ with the domain equal to $H^2(\R^N)$. Moreover, $-\Delta+V$ in $L^2(\R^N)$ is self-adjoint, its spectrum $\sigma(-\Delta+V)$ is bounded from below and the essential spectrum $\sigma_{ess} (-\Delta+V)$ is located in the interval $(\varrho(V_\infty),+\infty)$ where the so-called {\em asymptotic bottom} is given by
$$
\varrho(V_\infty):= \lim_{R\to +\infty} \essinf_{|x|>R} V_{\infty}(x).
$$
The spectrum of $-\Delta+V$ may have a nonempty part below the value $\varrho(V_\infty)$, which then contains eigenvalues only.  Therefore if $\varrho(V_\infty)>\lambda$, then 
$\sigma(-\Delta+V)\cap (-\infty, \lambda)$ consists of a finite number of eigenvalues (see Section \ref{sec-schr-op} for a more detailed discussion). 
This enables us to define the collective multiplicity of eigenvalues less than $\lambda$ as
\begin{equation*}
d^{-}(V,\lambda) :=\!\!\!\!\! \sum_{\mu \in \sigma (-\Delta+V) \cap (-\infty, \lambda)}\!\! \dim \Ker (-\Delta+V-\mu).
\end{equation*}

\indent Let us start with the so-called non-resonant case, i.e. when the linearizations of $f$ at zero and at infinity give potentials that do not resonate with the operator $-\Delta+V-\lambda$. We shall start with the following criterion on the existence of solutions of $(E)_\lambda$ and connecting orbits for the parabolic equations associated with $(E)_\lambda$ 
$$
u_t = \Delta u - V(x)u +\lambda u + f(x,u),\quad  x\in\R^N,\ t>0. \leqno{(P)_\lambda}
$$ 
\begin{theorem}\label{noneresonat-case-theorem}
Let us assume that $f$ satisfies $(f1)-(f3)$ and  there exist potentials $\alpha, \omega:\R^N\to\R$ satisfying (PT) such that
$$
\lim_{u\to\; 0} \frac{f(x,u)}{u}=\alpha (x) \quad\text{and}\quad
\lim_{|u|\to +\infty} \frac{f(x,u)}{u}=\omega (x)\quad \text{for a.e. $x\in\R^N$}.
$$
If $\lambda < \min\left\{ \varrho (V_\infty-a_\infty), 
\varrho(V_\infty-\alpha_\infty), \varrho  (V_\infty-\omega_\infty) \right\}$ is such that
$$
\lambda\not\in \sigma (-\Delta+V-\alpha) \cup \sigma(-\Delta+V-\omega)
\ \ \text{ and } \ \ d^-(V-\alpha,\lambda) \neq d^{-}(V-\omega,\lambda),
$$
then there exists a nonzero solution $\bar u\in H^2(\R^N)$ of $(E)_\lambda$ and a bounded solution $u:\R\to H^1(\R^N)$ of the equation $(P)_\lambda$ such that either $\bar u$ is in the $\alpha$-limit set of $u$ and 
$u(t)\to 0$ in $H^1(\R^N)$ as $t\to +\infty$, or 
$\bar u$ is in the $\omega$-limit set of $u$ and $u(t)\to 0$ in $H^1(\R^N)$ as $t\to -\infty$. Moreover, the set of all solutions to $(E)_\lambda$ is bounded in $H^1(\R^N)$.
\end{theorem}
\noindent Theorem \ref{noneresonat-case-theorem} is an extension of the result by Prizzi -- see \cite[Cor. 3.4]{Prizzi-FM}; the difference is that less restrictive assumptions on $\alpha$ and $\omega$ are considered, no differentiability of $f$ is assumed and a more general potential $V$ is possible, i.e. satisfying $(KR)$ instead of $(PT)$.\\
\indent In the resonant case, i.e. when $-\Delta +V-\lambda$ has nonempty kernel we shall assume that $f$ satisfies $(f2)$ and is "bounded" in the sense that there is $m\in L^{2}(\R^{N})$ such that
$$ 
|f(x,u)|\leq m(x)\quad \text{ for a.e.  $x\in\R^N$ and all $u\in\R$}. \leqno{(f1)'}
$$
\noindent Here we shall impose additional assumptions on $f$ at infinity, the so-called {\em Landesman-Lazer type} and {\em strong resonance} conditions. The {\em Landesman-Lazer type} conditions state that either
$$
\left\{
\begin{array}{c}
\check{f}_+(x) \geq 0 \ \mbox{ and } \ \hat{f}_-(x)\leq 0 \ \mbox{ for a.e. } \ x\in\R^N,\\
\mbox{there is a set of positive measure on which none of}\;\; \check{f}_+ \mbox{ and } \hat{f}_- \mbox{vanishes},
\end{array}
\right.
\leqno{(LL)_+}
$$
or
$$
\left\{
\begin{array}{c}
\hat{f}_+(x)\leq 0 \ \mbox{ and } \ \check{f}_- (x)\geq 0 \ \mbox{ for a.e. } \ x\in\R^N,\\
\mbox{there is a set of positive measure on which none of }\; \hat{f}_+ \mbox{ and }  \check{f}_- \mbox{ vanishes,}
\end{array} \right.
\leqno{(LL)_-}
$$
where $\hat{f}_\pm(x):= \limsup_{s \to \pm\infty} f(x,s)$ and  $\check{f}_\pm (x):= \liminf_{s \to \pm\infty} f(x,s)$ for $x\in\R^N$. 
We remark that the continuation property \cite{Gossez} for the eigenvalues  of the Schr\"odinger operator implies that the conditions $(LL)_\pm$ represent a particular  form of the original resonance assumptions introduced by Landesman and Lazer in \cite{MR0267269} (see \cite[Remark 1.3]{Cw-Kr-2019} for more details).

The so-called {\em sign} or {\em strong resonance conditions} state
$$
\left\{
\begin{array}{c}
f (x,s)s \geq 0,  \mbox{ for a.e. } x\in \R^N \mbox{ and all } s\in\R,\\
\mbox{there is a set of positive measure on which } \check{k}_+  \mbox{ and }  \check{k}_- \mbox{ are positive,}
\end{array}
\right.
\leqno{(SR)_+}
$$
or
$$
\left\{
\begin{array}{c}
f (x,s)s \leq 0,  \mbox{ for a.e. } x\in \R^N \mbox{ and all } s\in\R,\\
\mbox{there is a set of positive measure on which }\hat{k}_+  \mbox{ and } \hat{k}_- \mbox{ are negative,}
\end{array}
\right.
\leqno{(SR)_-}
$$
where $\check{k}_\pm(x) = \liminf_{s\to \pm\infty} f(x,s)s$ and $\hat{k}_\pm(x) = \limsup_{s\to \pm\infty} f(x,s)s$ for $x\in\R^N$.\\
\indent In the above setting we obtain the following existence result in the resonant case, i.e. when $\lambda$ is in the point spectrum of $-\Delta+V$.
\begin{theorem}\label{30042019-1204}
Let us assume that $\lambda$ is an eigenvalue of the operator $-\Delta+V$ such that $\varrho \big(V_{\infty}\big) > \lambda$ and that $f$ satisfies $(f1)$, $(f1)'$ and $(f2)$. If one of conditions $(LL)_+$, $(LL)_-$, $(SR)_+$ or $(SR)_-$ is satisfied, then the equation $(E)_\lambda$ admits a solution $u\in H^2(\R^N)$ and the set of all solutions to $(E)_\lambda$ is bounded.
\end{theorem}
In case of the existence of a trivial solution we get the following criterion.
\begin{theorem} \label{30042019-1206}
Let $f$ be a map satisfying $(f1)$, $(f1)'$, $(f2)$, $(f3)$. Assume also that  $f(x,0)=0$ for a.e. $x\in\R^N$ and there exist a potential $\alpha$ satisfying $(PT)$ such that 
$$ 
\lim_{u\to\; 0} \frac{f(x,u)}{u} = \alpha(x) \quad \text{for a.e. $x\in\R^N$}.
$$ 
If $\lambda < \min\{ \varrho(V_\infty), \varrho (V_{\infty} - \alpha_\infty), \varrho(V_\infty-a_\infty) \}$ is an eigenvalue of 
$-\Delta+V$ such that one of the following conditions is fulfilled\\[2pt]
\noindent\makebox[7mm][r]{$(i)$} \parbox[t][][t]{166mm}{either $(LL)_+$ or $(SR)_+$ holds and $d^-(V,\lambda) +\dim \Ker (-\Delta+V-\lambda) \neq d^{-} (V-\alpha,\lambda)$,}\\[2pt]
\noindent\makebox[7mm][r]{$(ii)$} \parbox[t][][t]{166mm}{either $(LL)_-$ or $(SR)_-$ holds and $d^-(V,\lambda) \neq d^{-} (V-\alpha,\lambda)$,}\\[2pt]
then there exists a nonzero solution $\bar u\in H^2(\R^N)$ of $(E)_\lambda$ and a bounded solution $u:\R\to H^1(\R^N)$ of 
such that either $\bar u$ is in the $\alpha$-limit set of $u$ and  $u(t)\to 0$ in $H^1(\R^N)$ as $t\to +\infty$, or 
$\bar u$ is in the $\omega$-limit set of $u$ and 
$u(t)\to 0$ in $H^1(\R^N)$ as $t\to -\infty$. 
\end{theorem}
\noindent As an immediate consequence we get the following straightforward criterion. 
\begin{corollary}\label{corollary-resonance}
Suppose that $V$, $f$ and $\lambda$ are as in Theorem \ref{30042019-1206} and, additionally, that $\alpha$ is a constant function equal to $\bar \alpha\in\R$. 
If one of the following conditions is satisfied\\[2pt]
\noindent\makebox[7mm][r]{$(i)$} \parbox[t][][t]{166mm}{if either $(LL)_+$ or $(SR)_+$ holds and $\bar\alpha<0$,}\\[2pt]
\noindent\makebox[7mm][r]{$(ii)$} \parbox[t][][t]{166mm}{if either $(LL)_-$ or $(SR)_-$ holds and $\bar\alpha>0$,}\\[2pt]
\noindent\makebox[7mm][r]{$(iii)$} \parbox[t][][t]{166mm}{if either $(LL)_+$ or $(SR)_+$ holds, $\bar\alpha>0$ and $\sigma(-\Delta+V)\cap (\lambda,\bar\alpha + \lambda)\neq \emptyset$,}\\[2pt]
\noindent\makebox[7mm][r]{$(iv)$} \parbox[t][][t]{166mm}{if either $(LL)_-$ or $(SR)_-$ holds, $\bar\alpha<0$ and $\sigma(-\Delta+V)\cap (\bar\alpha+\lambda,\lambda)\neq \emptyset$,}\\[2pt]
then the assertion of Theorem \ref{30042019-1206} holds.
\end{corollary}
\noindent Theorem \ref{30042019-1206} provides the existence of nontrivial solution for $(E)_\lambda$. It is showed in \cite{Cw-Kr-2019} that under the Landesman-Lazer or strong resonance conditions, if $\lambda=\lambda_0$ is an isolated eigenvalue (as in Theorem \ref{30042019-1204} or Theorem \ref{30042019-1206}), then $\lambda_0$ is an asymptotic bifurcation point (or a point of bifurcation from infinity) for the parameterized family of equations $(E)_\lambda$. In this context our results show that the large solutions for the equation $(E)_\lambda$, with respect to the $H^1(\R^N)$ norm, appear also in the case $\lambda = \lambda_0$.\\
\indent Our approach is based on the Rybakowski version of Conley index theory for semiflows on non-compact spaces -- see \cite{rybakowski} and \cite{rybakowski-TAMS}. In this respect we follow the pioneering work \cite{rybakowski-1985} as well as \cite{Prizzi-FM}, \cite{Kokocki}, \cite{Kokocki1} or \cite{Cw-Maciejewski}, where the techniques involving irreducible sets to problems on bounded domains were studied. We also refer the reader to the recent paper \cite{rybakowski-new} where the irreducibility method were used to find specific full nonconstant solutions for the partial functional differential equation containing delay terms.

Here we consider the semiflow $\Phi$ on the space $H^1(\R^N)$ generated by the nonlinear equation
\begin{equation}\label{abstract-nonlinear-eq}
\dot u(t) = - A u(t) + F(u(t)), \quad t>0,
\end{equation}
with $A:=- \Delta + V-\lambda$ and the Nemytskii operator $F:H^1(\R^N)\to L^2 (\R^N)$ determined by $f$. The admissibility property of the semiflow $\Phi$ and related homotopies, that are necessary for application of Conley index, is based on tail estimate techniques (see e.g. \cite{Schmitt_Wang} and \cite{Prizzi-FM}). In order to compute the Conley indices at resonance, we shall exploit, after \cite{Kokocki} (where bounded domains are studied) and \cite{Cw-Kr-2019} (bifurcations for problems in $\R^N$), the fact that both Landesman-Lazer and sign conditions have geometric consequences in the phase space of the parabolic equation \eqref{abstract-nonlinear-eq}. We show the boundedness of the maximal invariant set $K_\infty(\Phi)$, consisting of all $\bar u\in H^1(\R^N)$ such that there exists a full bounded solution of \eqref{abstract-nonlinear-eq} with $u(0)=\bar u$. Next we compute the Conley indices of the set $K_\infty(\Phi)$ and the index of the isolated trivial (zero) solution. This allows us to deduce the existence of a nonzero full solution of \eqref{abstract-nonlinear-eq}, which by the gradient-like structure of the semiflow $\Phi$ gives the existence of nontrivial solution of $(E)_\lambda$.\\
\indent The paper is organized as follows. Section 2 provides basic facts on Conley-Rybakowski index theory. In Section 3 we recall basic properties of the operator $-\Delta+V$ and derive an inequality that we shall use for tail estimates. Section 4 deals with the definition and properties of the semiflow $\Phi$ determined by $(P)_\lambda$ and its parameterized variants. In particular tail estimates are used to show the admissibility of bounded subsets of the phase space $H^1(\R^N)$.
Section 5 is devoted to the proof of Theorem \ref{noneresonat-case-theorem}, which is based on the mentioned Conley index formulas and the properties of irreducible sets. Finally in Section 6 we shall use the geometric implications of assumptions $(LL)_\pm$ and $(SR)_\pm$ in order to prove Theorems \ref{30042019-1204} and \ref{30042019-1206} as well as Corollary \ref{corollary-resonance}.


\section{Preliminaries on Conley index due to Rybakowski}
In this section we will briefly recall a version of Conley index for the non-compact metric spaces due to Rybakowski (see \cite{rybakowski} or \cite{rybakowski-TAMS}). Let $\Phi\colon [0,+\infty)\times X\to X$ be a  semiflow on a complete metric space $X$. 
A continuous map $u\colon J\to X$, where $J\subset \R$ is an interval, is a {\em solution of $\Phi$} if $u(t+s) = \Phi_t(u(s))$ for all $t\geq 0$ and $s\in J$ such that $t+s\in J$. 
Given $N\subset X$, we define the \emph{invariant part} of the set $N$ by
$$\mathrm{Inv}_\Phi (N) := \{x\in X \ | \ \text{there is a solution}\;u\colon \R\to N\;\text{such that}\;u(0)=x\}.$$
We say that a set $K\subset X$ is a {\em $\Phi$-invariant} or  {\em invariant} (w.r.t.\ $\Phi$) if  $\mathrm{Inv}_\Phi (K)=K$. If there exists an {\em isolating neighborhood} of $K$, i.e. the set $N\subset X$ such that  $K=\mathrm{Inv}_\Phi (N) \subset \mathrm{int}\,N$, then  $K$ is called an {\em isolated invariant} set.\\
\indent  A set $N\subset X$ is \emph{$\Phi$-admissible} or {\em admissible} (w.r.t. $\Phi$) if, for any sequences $(t_n)$ in $[0,+\infty)$, $(x_n)$ in $X$ such that $t_n\to +\infty$ and $\{\Phi(t,x_{n}) \ | \ t\in[0,t_{n}]\}\subset N$, the sequence of the end-points $\left( \Phi_{t_n}(x_n) \right)$ has a convergent subsequence.\\
\indent Suppose that $\{\Phi^{(s)}\}_{s\in \Lambda}$, where $\Lambda$ is a metric space, is a family of semiflows on $X$ such that the map $[0,+\infty)\times X\times \Lambda\ni (t,x,s)\mapsto\Phi^{(s)}_t(x)$ is continuous. A set $N\subset X$ is
{\em admissible} with respect to the family of semiflows $\{\Phi^{(s)}\}_{s\in\Lambda}$ if, for any sequences $(t_n)$ in $[0,+\infty)$, $(x_n)$ in $X$ and $(s_n)$  such that $t_n\to +\infty$, $s_n\to s_0$ in $\Lambda$ and $\{\Phi^{(s_n)}(t,x_{n}) \ | \ t\in [0,t_n]\}\subset N$, the sequence $(\Phi^{(s_n)}_{t_n}(x_n))$ has a convergent subsequence.\\
\indent Let ${\mathcal I}(X)$ be the family of all pairs $(\Phi, K)$ such that $\Phi$ is a semiflow on $X$ and $K\subset X$ is an isolated invariant set   having an admissible isolating neighborhood w.r.t.\ $\Phi$. Given $(\Phi,K)\in {\mathcal I}(X)$, we can choose a pair of closed sets $(B,B^-)$, where $B$ is an isolating block of $K$ relative to $\Phi$ and $B^-$ is its exit set (see \cite{rybakowski}). Then the Conley {\em homotopy index} $h(\Phi,K)$ of $K$ relative to $\Phi$ is defined as the homotopy type of the quotient space
$$h(\Phi, K):=[(B/B^-, [B^-])],$$
provided $B^-$ is a nonempty set. If $B^-=\emptyset$ then we put
$h(\Phi, K):=[(B\,\dot\cup\, \{ a\}, a)]$, where $\{a\}$ is the one point space disjoint with $B$. In particular, $h(\Phi, \emptyset)=\overline 0$, where $\overline 0:=[(\{a\},a)]$.

Let us enumerate several important properties of the homotopy index.
\begin{enumerate}
	\item[(H1)]
	For any $(\Phi, K)\in {\mathcal I}(X)$, if $h(\Phi, K)\neq \overline{0}$, then $K\neq \emptyset$;
	\item[(H2)] if $(\Phi, K_1), (\Phi, K_2)\in {\mathcal I}(X)$ and $K_1\cap K_2=\emptyset$, then $(\Phi, K_1\cup K_2)\in {\mathcal I}(X)$
	and $h(\Phi, K_1\cup K_2) = h(\Phi, K_1)\vee h(\Phi, K_2)$;
	\item[(H3)] for any $(\Phi_1, K_1)\in {\mathcal I}(X_1)$ and $(\Phi_2, K_2)\in {\mathcal I}(X_2)$, $(\Phi_1\times \Phi_2, K_1\times K_2)\in {\mathcal I} (X_1\times X_2)$
	and $h(\Phi_1\times \Phi_2, K_1\times K_2) = h(\Phi_1, K_1) \wedge h(\Phi_2, K_2)$;
	\item[(H4)] if the family of semiflows $\{\Phi^{(s)}\}_{s \in [0,1]}$ is continuous and there exists an admissible (with respect to this family) $N$ such that $K_s = \mathrm{Inv}_{\Phi^{(s)}} (N) \subset \mathrm{int}\ N$ for $s\in [0,1]$, then
$$h(\Phi^{(0)}, K_0) = h(\Phi^{(1)}, K_1).$$
\end{enumerate}
We recall that, if $u\colon \R\to X$ is a solution of the semiflow $\Phi$, then the {\em $\omega$-limit set} of $u$ is defined by
$$\omega(u):=\{x=\lim_{n\to\infty} u(t_n)\mid \text{for some } \ t_n\to +\infty\}.$$
Similarly, the {\em $\alpha$-limit set} of $u$ is defined by
$$\alpha (u):=\{x=\lim_{n\to\infty} u(t_n)\mid \text{for some }\ t_n\to -\infty\}.$$
It is known that if $N\subset X$ is $\Phi$-admissible and $u:\R\to N$ is a solution of $\Phi$, then $\alpha(u)$ and $\omega(u)$ are both nonempty compact invariant sets. \\
\indent An isolated invariant with respect to the semiflow $\Phi$ set $K$ is called {\em irreducible} if there are no isolated invariant sets $K_1$ and $K_2$ such that $K=K_1 \cup K_2$, $K_1\cap K_2=\emptyset$  and
$$
h(\Phi, K_1)\neq \overline 0 \quad\text{ and }\quad h(\Phi, K_2)\neq \overline 0.
$$
It is known that if $K$ is connected or $h(\Phi, K)=\overline{0}$ or $h(\Phi, K)=\Sigma^k$, where $k\geq 0$ is an integer \!\footnote{We recall that $\Sigma^k:=[(S^k,s_0)]$ is the homotopy type of the pointed $k$-dimensional sphere $S^{k}:=\{x\in\R^{k+1} \ | \ |x|=1\}$.}, then $K$ is irreducible (see \cite[Theorem 1.11.6]{rybakowski}).
\begin{theorem}{\em (see \cite[Theorem 1.11.5]{rybakowski})} \label{rybakowski-irreducible}
Assume that $K_0\subset K\subset X$ is a pair of the isolated invariant sets such that $K$ is irreducible and
$$
\overline 0 \neq h(\Phi, K_0)\neq h(\Phi, K)\neq \overline 0.
$$
Then there exists a full solution $u:\R\to K$ such that $u(\R)\not\subset K_0$ and either $\alpha(u)\subset K_0$ or $\omega(u)\subset K_0$.
\end{theorem}
Recall that a continuous function $\mathcal{E}:X\to\R$ is called a {\em Liapunov functional} for the semiflow $\Phi$ if  for any 
$x\in X$, the function $t \mapsto \mathcal{E} (\Phi_t(x))$ is nonincreasing on $(0,+\infty)$. The semiflow $\Phi$ is called {\em gradient-like} with respect to $\mathcal{E}$ if $\mathcal{E}$ is a Lyapunov functional for $\Phi$ and, for any nonconstant solution $u:\R\to X$ of $\Phi$, the function $t\mapsto \mathcal{E}(\Phi_t(x))$ is not constant. It is known that if $\Phi$ is gradient-like with respect to some Lyapunov functional for $\Phi$, then, for any solution $u: \R\to X$ of  $\Phi$ with relatively compact $u(\R)$, the sets $\alpha(u)$ and $\omega(u)$ are nonempty and disjoint sets containing only equilibria of $\Phi$, i.e. elements $x\in X$ such that $\Phi_t(x)=x$ for all $t\geq 0$ (see \cite[Theorem 2.5.4]{rybakowski}).

\section{Properties of the Schr\"odinger operator}\label{sec-schr-op}
Let us consider an operator $A$ on the space $X:=L^2(\R^N)$ given by the formula
\begin{equation}\label{def-a}
A u: = -\Delta u +V u - \lambda u,  \quad \ u\in D(A):=H^2(\R^N),
\end{equation}
where $\lambda< \varrho(V_\infty)$.
Clearly $A$ is a perturbation of $A_0:H^2(\R^N) \to L^2(\R^N)$, given by $A_0 u:=-\Delta u$, $u\in H^2(\R^N)$, by the multiplication operator $V: H^2(\R^N) \to L^2(\R^N)$. It is known that $A_0$ is self-adjoint and $-A_0$ generates an analytic $C_0$-semigroup of bound linear operators on $X$. Since the potential $V$ is of Kato-Rellich type
we infer that $H^2(\R^N)$ continuously embeds into $L^q(\R^N)$ for $q=+\infty$ if $N=1,2,3$,  for arbitrary $q\geq 2$ if $N=4$ and $q=2N/(N-4)$ if $N\geq 5$ (see \cite[Corollary 9.13]{MR2759829}). This implies that the multiplication operator $V$ is well-defined and $A_0$-bounded with zero $A_0$-bound, i.e., for any $\varepsilon>0$, there exists $C_\varepsilon>0$ such that
$$
\| V u \|_{L^2} \leq \varepsilon \| A_0 u\|_{L^2}+ C_\varepsilon \|u\|_{L^2} \quad \text{for all } \ u\in H^2(\R^N).
$$
Therefore, by the Rellich-Kato theorem (see e.g. \cite[Th. 13.5]{Hislop-Sigal}), the perturbed operator $A=A_0+V-\lambda I$ is self-adjoint and, by \cite[Th. 2.10]{Engel}, the operator $-A$ generates an analytic $C_0$-semigroup $\{S(t)\}_{t\geq 0}$. Moreover, it appears that the multiplication operator $V_0:H^2(\R^N)\to L^2(\R^N)$ is $(A_0+V_\infty)$-compact, which due to the Weyl theorem (see e.g. \cite[Th. 14.6]{Hislop-Sigal}) implies that the essential spectra of $A_0+V_\infty$ and $A$ are equal, i.e.
$$
\sigma_{ess} (-\Delta+V-\lambda) = \sigma_{ess}(-\Delta+V_\infty-\lambda).
$$
By the Persson formula (see e.g. \cite[Th. 14.11]{Hislop-Sigal} or \cite{Persson}) for the bottom of essential spectrum we get 
\begin{align*}
\inf\sigma_{\mathrm{ess}}(-\Delta+V_\infty) = \sup_{R>0}\ \inf\left\{ (-\Delta\phi+V_\infty\phi,\phi)_{L^2}\ | \ \phi\in C^{\infty}_{0}(\R^{N}\setminus D(0,R)), \, \|\phi\|_{L^2}=1 \right\},
\end{align*}
where $D(0,R):=\{x\in\R^N \mid |x|\leq R \}$.
 Therefore the essential part of the spectrum of the operator $A$ satisfies $\inf\sigma_{\mathrm{ess}}(A) \ge \varrho( V_{\infty})-\lambda$. Hence the part of the spectrum $\sigma(A)$ contained in the interval $(-\infty, \varrho (V_{\infty})-\lambda)$ consists of isolated eigenvalues of finite multiplicity, that in turn satisfy the estimate 
\begin{align}\label{est-eigen}
|u(x)|\leq Ce^{-\delta |x|} \qquad \text{ for a.e. } \ x\in\R^{N},
\end{align}
where $\delta=\delta(u)>0$ and $C=C(u)>0$ are constants -- see \cite[Th. C.3.4]{Simon}. \\
\indent Since $A$ is self-adjoint, by the spectral theorem, there are closed mutually orthogonal subspaces $X_-$, $X_0$, $X_{+}$ of $X=L^2(\R^N)$ such that $X=X_- \oplus X_0 \oplus X_+$ and
\begin{equation*}
\sigma (A|X_-) = \sigma(A)\cap (-\infty,0), \  \ \ X_0 = \Ker \, A, \ \ \  \sigma (A|X_+) = \sigma(A)\cap (0,+\infty).
\end{equation*}
It is known that $\dim\, X_- <  \infty$ and $\dim X_0 <\infty$. We denote by $Q_{-}, P, Q_+$,  be the orthogonal projections onto $X_-$, $X_0$ and  $X_+$, respectively, and we write $Q:=Q_{+}+Q_{-}$. By the semigroup theory we know that 
\begin{equation*}
S(t)Q_{\pm} u = Q_{\pm} S(t) u,\quad u\in X, \ t\ge 0,
\end{equation*}
and there exist $K, \rho_+ >0$ such that
\begin{equation}\label{X-plus-semigroup-ineq}
\|S(t)u\|_{H^1} \leq K t^{-1/2} e^{-\rho_+ t}\|u\|_{L^2} \ \ \mbox{ for all } \ \ u\in X_+, \ t>0.
\end{equation}
Obviously, the restricted semigroup $\{ e^{-tA}|_{X_-}\}_{t\geq 0}$ can be extended to a $C_0$-group on $X_{-}$ and there exist $\rho_->0$ such that
\begin{equation}\label{X-minus-group-ineq}
\|S(t) u\|_{L^2} \leq e^{\rho_- t}\|u\|_{L^2}, \quad  u\in X_-, \ t\leq 0.
\end{equation}
The following theorem says that the spectral properties enable us to determine the Conley index of zero in the linear case.
\begin{theorem} \label{linear-conley-index}{\em (see \cite[Ch. I, Th. 11.1]{rybakowski})}
The restriction of the semigroup $\{S(t)_{|Y}\}_{t\ge 0}$ to $Y:=H^{1}(\R^{N})\cap (X_{+}\oplus X_{-})$ determines the semiflow on the space $Y$ equipped with the norm $\|\cdot\|_{H^{1}}$ {\em(}note that $Y=H^1(\R^N)$ if $\lambda\not\in\sigma(-\Delta+V)${\em)}. Furthermore the set $K_0=\{0\}$ is the maximal bounded invariant set of the semiflow, $(S(t)_{|Y}, K_0)\in {\mathcal I}(X)$ and $$
h(S(t)_{|Y}, K_0)=\Sigma^{\dim X_-}.
$$
\end{theorem}
Let us consider the semilinear differential equation 
\begin{equation}\label{eq-diff-g1}
\dot u(t) = - A u(t) + g(t),\quad t>0,
\end{equation}
where $g:[0,\infty)\to X$ is a continuous function and define the cut-off map $\phi_{n}(x) := \phi(|x|^{2}/n^{2})$ for $x\in \R^{n}$, where the profile $\phi:[0,\infty)\to [0,1]$ is a smooth function such that $\phi(x) = 0$ for $x\in[0,1]$ and $\phi(x) = 1$ for $|x|\ge 2$. 
The following version of the estimates for the solutions of \eqref{eq-diff-g1} can be obtained along the lines of \cite[Lem. 3.6]{Cw-Kr-2019} and \cite[Lem. 5.3]{Cw-Luk-2021}.
\begin{proposition}\label{classic_tail_estimates}
There exist a sequence $(\beta_{n})$ with $\beta_{n}\to 0^+$ as $n\to\infty$ and $n_0\geq 1$, such that if $u\in C((t_0,t_{1}),H^2(\R^{N}))\cap C^1((t_0,t_{1}),L^2(\R^{N}))$ is a classical solution of the equation \eqref{eq-diff-g1}, then
\begin{align}\label{eq-p-1bm}
\frac{1}{2}\frac{d}{dt}\int_{\R^{N}}\phi_{n}|u(t)|^{2}\,dx \leq -\int_{\R^{N}}(V_{\infty}(x)-\lambda)\phi_{n}|u(t)|^{2}\,dx  + \gamma_{n}\|u(t)\|_{H^{1}}^{2} + \langle \phi_{n}u(t),g(t)\rangle_{L^{2}}
\end{align}
for $t\in(t_{0},t_{1})$ and $n\ge n_0$. The above sequence $(\beta_n)$ and integer $n_0\ge 1$ depend only on the potential $V$. 
\end{proposition}
\begin{proof}
We begin with the observation that
\begin{equation}
\begin{aligned}\label{ineq-k1}
& \frac{1}{2}\frac{d}{dt}\langle\phi_{n}u(t), u(t)\rangle_{L^{2}} = \langle\phi_{n}u(t),\dot u(t)\rangle_{L^{2}} 
= \langle\phi_{n}u(t),- A u(t) + g(t)\rangle_{L^{2}} \\
&\quad = \langle\phi_{n}u(t),\Delta u(t) \rangle_{L^{2}} \!+\! \langle\phi_{n}u(t), -V u(t) \!+\! \lambda u(t)\rangle_{L^{2}} \!+\! \langle\phi_{n}u(t),g(t)\rangle_{L^{2}} \\
&\quad =: I_{1}(t) + I_{2}(t) + I_{3}(t)
\end{aligned}
\end{equation}
for $t\in[t_{0},t_{1}]$. Then we have the following estimates
\begin{equation}
\begin{aligned}\label{ineq-k2}
I_{1}(t) & = \langle\phi_{n}u(t),\Delta u(t) \rangle_{L^{2}} = -\langle\nabla(\phi_{n}u(t)),\nabla u(t) \rangle_{L^{2}} \\
&= -\!\int_{\R^{N}} \phi_{n}(x)|\nabla u(t)|^{2} \,dx -\frac{2}{n^{2}}\int_{\{n\leq|x|\leq n\sqrt{2}\}} \!u(t)\phi'\left(\tfrac{|x|^{2}}{n^{2}}\right)\, x \cdot \nabla u(t)\,dx \\
&\leq \frac{2\sqrt{2}}{n}L_{\phi}\|\nabla u(t)\|_{L^{2}}\|u(t)\|_{L^{2}} \leq \frac{2\sqrt{2}}{n}L_{\phi}\|u(t)\|_{H^{1}}^{2},
\end{aligned}
\end{equation}
where $L_{\phi}:=\sup_{t\ge 0} |\phi'(t)|$. On the other hand we have 
\begin{equation}\label{ineq-k3}
\begin{aligned}
I_{2}(t)& =\langle\phi_{n}u(t), -Vu(t) +\lambda u(t)\rangle_{L^{2}}  = -\int_{\R^{N}}(V_{\infty}(x)-\lambda)\phi_{n}|u(t)|^{2}\,dx - \int_{\R^{N}}V_{0}(x)\phi_{n}|u(t)|^{2}\,dx \\
& \leq -\int_{\R^{N}}(V_{\infty}(x)-\lambda)\phi_{n}|u(t)|^{2}\,dx + \|\phi_{n}V_{0}\|_{L^{p}}\|u(t)\|_{L^{2p/(p-1)}}^{2} \\
& \leq -\int_{\R^{N}}(V_{\infty}(x)-\lambda)\phi_{n}|u(t)|^{2}\,dx + \|\phi_{n}V_{0}\|_{L^{p}}\|u(t)\|_{H^{1}}^{2},
\end{aligned}
\end{equation}
where, in the last estimate, we used the Sobolev embedding $H^{1}(\R^{N})\subset L^{2p/(p-1)}(\R^{N})$, which is valid for $p>2$ satisfying the condition $(KR)$. 
Therefore, combining \eqref{ineq-k1}, \eqref{ineq-k2} and \eqref{ineq-k3} gives \eqref{eq-p-1bm} with 
\begin{align*}
\beta_{n}:= \frac{4\sqrt{2}}{n}L_{\phi} + 2\|\phi_{n}V_{0}\|_{L^{p}},\quad n\ge 1.
\end{align*}
Since $V_{0}\in L^{p}(\R^{N})$, we have $\beta_{n}\to 0^+$ as $n\to\infty$ and the proof of the proposition is completed. 
\end{proof}

\section{Semiflow and admissible homotopy}
Consider the differential equation
\begin{equation}\label{abs-eq}
\dot u (t)= - Au(t)+F(u(t)), \quad t>0,
\end{equation}
where $A$ is defined by \eqref{def-a} and the nonlinear map $F:H^1(\R^N)\to L^2(\R^N)$ is the Nemytskii operator determined by $f$, i.e. it is given by \begin{equation}\label{Nemytzki_f}
[F(u)](x):= f(x,u(x)) \quad \mbox{ for a.e. } \ x\in\R^N \text{ and all }  u\in H^1(\R^N).
\end{equation}
\begin{remark}
{\em From the assumptions $(f1)$ and $(f2)$ it follows that $f$ is a Careth\'{e}odory function. Therefore, for any $u\in H^{1}(\R^{N})$, the composition in \eqref{Nemytzki_f} is a measurable function. On the other hand, by $(f2)$ and the H\"older inequality, for any $u,v\in H^{1}(\R^{N})$, we have
\begin{equation}
\begin{aligned}\label{eq-lip}
\|F(u) - F(v)\|_{L^{2}} & \leq \|(l_{\infty}+l_{0})(u-v)\|_{L^{2}} \\
& \leq \|l_{\infty}\|_{L^{\infty}}\|u-v\|_{L^{2}}+ \|l_{0}\|_{L^{r}}\|u-v\|_{L^{2r/(r-2)}} \\
& \leq \left(\|l_{\infty}\|_{L^{\infty}}+ \|l_{0}\|_{L^{r}}\right)\|u-v\|_{H^{1}},
\end{aligned}
\end{equation}
where we used the Sobolev embedding $H^{1}(\R^{N})\subset L^{2r/(r-2)}(\R^{N})$, which holds true as $r\in(2,\infty)$ is restricted in $(PT)$. This implies that $F$ is a Lipschitz map. Furthermore, by $(f1)$ and \eqref{eq-lip}, we have
\begin{align}\label{eq-s-lin}
\|F(u)\|_{L^{2}} \leq \left(\|l_{\infty}\|_{L^{\infty}}+ \|l_{0}\|_{L^{r}}\right)\|u\|_{H^{1}} + \|c\|_{L^{2}}
\end{align}
for $u\in H^{1}(\R^{N})$, which gives the sublinear growth of $F$. \hfill $\square$}
\end{remark}
Since the operator $A$ is sectorial and $F$ satisfies \eqref{eq-lip} and \eqref{eq-s-lin}, it follows (see e.g. \cite[Theorem 3.3.3]{Henry}) that, for any $\bar u\in H^1(\R^N)$ and $t_0\in\R$, there exists a unique continuous function $u=u(\,\cdot\,; \bar u,t_0):[t_0,\infty) \to H^1(\R^N)$ such that $u(t_0)=\bar u$,
$u\in C((t_0,+\infty),H^2(\R^{N}))\cap C^1((t_0,+\infty),L^2(\R^{N}))$ and \eqref{abs-eq} holds for all $t>t_0$. Consequently the equation \eqref{abs-eq} determines the semiflow $\{\Phi_t\}_{t\geq 0}$ on the space $H^1(\R^N)$, given by
\begin{equation*}
\Phi_t (\bar u):= u (t; \bar u,0), \quad \bar u\in H^1(\R^N), \ t\ge 0.
\end{equation*}
It appears that the semiflow $\Phi$ is gradient-like.
\begin{proposition}\label{prop-fradc-d-c}
Let ${\mathcal E}:H^1(\R^N)\to\R$ be given by
\begin{equation}\label{eq-lyap}
{\mathcal E}(v):= \frac{1}{2} \int_{\R^N} |\nabla v(x)|^2 + (V(x)-\lambda)|v(x)|^2 \, d x + \int_{\R^N} {\mathcal F}(x,v(x))\, d x, \quad v\in H^1(\R^N),
\end{equation}
where ${\mathcal F}:\R^N\times \R\to\R$ is given by ${\mathcal F} (x,v):= \int_{0}^{v} f(x,w)\, dw$ for $x\in\R^N$ and $v\in\R$.
The function  ${\mathcal E}$ is a well-defined continuous Lyaponov functional for the semiflow $\Phi$, i.e. if $u:[t_0,t_{1}) \to H^1(\R^N)$ is a solution of the semiflow $\Phi$, then  
\begin{align*}
\frac{\d}{\d t} \mathcal{E} (u(t))  =-\int_{\R^N} |\dot u(t)|^{2} \,d x \quad  \text{for} \ \ t\in (t_0, t_1). 
\end{align*}
\end{proposition}
\begin{proof} If $u:[t_0,t_1)\to H^1(\R^N)$ is a solution of \eqref{abs-eq}, then both $\dot u$ and $F\circ u$ are continuous functions on $(t_0,+\infty)$ with values in $L^2(\R^N)$, and consequently $A\circ u = -\dot u + F \circ u$ is continuous, which implies that $u\in C((t_0,t_1), H^2(\R^N))$. Fix $t\in (t_0,t_1)$ and take any small $h\neq 0$, then, using the fact that $u\in C((t_0,t_1), H^2(\R^N))$, we get
\begin{align*}
& \frac{1}{h} \big( \|\nabla u(t+h)\|_{L^2}^{2}- \|\nabla u(t) \|_{L^2}^{2} \big) \\ 
& \quad = \langle \nabla u(t+h), \nabla\big(h^{-1} (u(t+h)-u(t))\big)\rangle_{L^2} + \langle  \nabla\big(h^{-1} (u(t+h)-u(t))\big), \nabla u(t)\rangle_{L^2} \\
& \quad = - \langle \Delta u(t+h), h^{-1}(u(t+h)-u(t)) \rangle_{L^2}- \langle h^{-1}(u(t+h)-u(t)), \Delta u(t)\rangle_{L^2}\\
& \quad \to 2 \langle -\Delta u(t), \dot u(t)\rangle_{L^2}   \ \text{ as } h\to 0.
\end{align*}
In a similar manner
$$
\frac{d}{d t} \int_{\R^N} (V(x)-\lambda)|u(t)|^2 = 2\int_{\R^N} (V(x)-\lambda) u(t) \dot u(t)\, dx.
$$
By standard arguments
\begin{align*}
& \frac{1}{h}\int_{\R^N} (\mathcal{F}(x,u(t+h)(x))-
\mathcal{F}(x, u(t)(x))\, dx \\
 & \quad = \frac{1}{h} \int_{\R^N} g_n(x) \big(u(t+h)(x)-u(t)(x)\big)\, dx
  \to \langle F(u(t)), \dot u(t)\rangle_{L^2} \ \text { as } h\to 0^+,
\end{align*}
where  $g_h (x):=f \big(x, u(t)(x)+\theta_{x} (u(t+h)(x)-u(t)(x)) 
\big)$, with $\theta_x\in (0,1)$ from the Lagrange mean value theorem applied for fixed $x\in \R^N$, and we use the fact that $g_h \to F(u(t))$  in $L^2(\R^N)$  as $h\to 0$. The latter follows from the following inequality implied by $(f2)$
$$
|g_h(x)-f(t,u(t)(x))| \leq l(x) |u(t+h)(x)-u(t)(x)| \ \text{ for a.e. } x\in\R^N,
$$
which implies $\|g_h-F(u(t))\|_{L^2}\leq C (\|l_\infty\|_{L^\infty}+ \|l_0\|_{L^{r}})\|u(t+h)-u(t)\|_{H^1}$ as in \eqref{eq-lip}.\\
\indent Combining the above three equalities together we obtain
$$
(\mathcal{E} \circ u)'(t) = \langle -\Delta u(t) + (V-\lambda)u(t) - F(u(t)), \dot u(t)\rangle_{L^2}  = - \| \dot u(t)\|_{L^2(\R^N)}^{2},
$$
which proves the assertion.
\end{proof}

\indent Let us also consider a parameterized by $s\in[0,1]$ family of differential equations
\begin{equation}\label{eq-diff-g}
\dot u(t) = - A u(t) + G(u(t),s),\quad t>0,
\end{equation}
where $G: H^1(\R^N) \times [0,1] \to L^2(\R^N)$ is a continuous map satisfying the following conditions\\[5pt]
\noindent\makebox[9mm][l]{$(G1)$}\parbox[t][][t]{166mm}{the map $G$ is locally Lipschitz, that is, for any $R>0$, there is $L_{R}>0$ such that
\begin{equation*}
\|G(u,s)-G(v,s)\|_{L^{2}}\leq L_{R}\|u-v\|_{H^{1}}
\end{equation*}
for any $s\in[0,1]$ and $u,v \in H^1(\R^N)$ with $\|u\|_{H^{1}}$, $\|v\|_{H^{1}} \leq R$;}\\[5pt] 
\noindent\makebox[9mm][l]{$(G2)$}\parbox[t][][t]{166mm}{the map $G$ has sublinear growth, that is, there is $C>0$ such that
\begin{equation*}
\|G(u,s)\|_{L^{2}}\leq C(\|u\|_{H^{1}} + 1)
\end{equation*}
for any $s\in[0,1]$ and $u\in H^1(\R^N)$.}\\[5pt]
Given $s\in [0,1]$, we denote by $\Psi^{(s)} = \big\{\Psi_{t}^{(s)} \big\}_{t\geq 0}$ the semiflow on $H^1(\R^N)$ determined by the differential equation \eqref{eq-diff-g}. We shall use the following continuation property -- see e.g. \cite[Th. 3.2]{Cw-Luk-2021}.
\begin{proposition}\label{semiflow-continuity-prop}
The following assertions hold. \\[3pt]
\noindent\makebox[5mm][r]{$(i)$} \parbox[t][][t]{166mm}{If $\bar u_n \to \bar u_0$ in $H^1(\R^N)$ and $s_n\to s_0$, then $\Psi_{t}^{(s_n)}(\bar u_n) \to \Psi_{t}^{(s_0)}(\bar u_0)$ in $H^1(\R^N)$ for all $t\ge 0$.}\\[3pt]
\noindent\makebox[5mm][r]{$(ii)$} \parbox[t][][t]{166mm}{If $(\bar u_n)$ is a bounded sequence in $H^1(\R^N)$ such that $\bar u_n\to \bar u_0$ in $L^2(\R^N)$ for some $\bar u_0\in H^1(\R^N)$, then $\Psi_{t}^{(s_n)} (\bar u_n) \to \Psi_{t}^{(s_0)}(\bar u_0)$ in $H^1(\R^N)$ uniformly with respect to $t$ in compact subsets of $(0,\infty)$.}
\end{proposition}

The following so-called tail estimates will be used in order to establish compactness properties of the family of semiflows  $\{\Psi^{(s)}\}_{s\in [0,1]}$ -- compare \cite[Prop. 3.7 and Cor. 3.9]{Cw-Kr-2019}.
\begin{proposition}\label{prop-tail} 
Let us assume that there is a potential $a=a_\infty+a_0$, where $a_\infty\in L^\infty (\R^N)$ and $a_0$ is a finite sum of potentials belonging to $L^p (\R^N)$ with $p$ as in the condition $(KR)$ (possibly with different exponents $p$ for each component) and $b\in L^2(\R^N)$ such that, for any $s\in [0,1]$ and $u\in H^1(\R^N)$, the following inequality is satisfied 
\begin{equation}\label{ineq-11ab}
u(x) \cdot G(u,s)(x) \leq a(x)|u(x)|^2 + b(x)|u(x)| \ \text{ for a.e. } \ x\in\R^N.
\end{equation}
If $\varrho(V_{\infty}-a_{\infty})>\lambda$  then, for any $R>0$, there exists $\ve>0$, a sequence $(\gamma_n)$ with $\gamma_n\to 0^+$ as $n\to\infty$ and $n_0\geq 1$, such that if $u:[t_0,t_1]\to H^1(\R^N)$ is a solution of the equation \eqref{eq-diff-g} for some $s\in [0,1]$ and
\begin{equation*}
\|u(t)\|_{H^1} \leq R \quad \text{for all} \ \ t\in [t_0, t_1],
\end{equation*}
then, for all $n\geq n_0$, we have
$$
\int_{\R^{N}}\phi_{n}|u(t)|^{2}\,dx \leq e^{-2\ve (t-t_0)} \|u(t_0)\|_{L^2}^2 + \gamma_n,\quad t\in[t_{0},t_{1}].
$$
The above constant $\ve>0$, the sequence $(\gamma_n)$ and the integer $n_{0}\ge 1$ depend only on $R>0$, the dimension $N\ge 1$ and the potentials $a$, $b$ and $V$.
\end{proposition}
\begin{proof} We shall carry out the proof with assuming that $a_0$ is in $L^p(\R^N)$ with $p$ like in the condition $(KR)$, an adaptation of the below estimates to the case where $a_0$ is a sum of such functionals is straightforward.\\
\indent By Proposition \ref{classic_tail_estimates}, there exist a sequence $(\beta_{n})$ with $\beta_{n}\to 0^+$ as $n\to\infty$ such that
\begin{align}\label{eq-p-1kl}
\frac{1}{2}\frac{d}{dt}\int_{\R^{N}}\phi_{n}|u(t)|^{2}\,dx  \!\leq\! -\!\int_{\R^{N}}(V_{\infty}-\lambda)\phi_{n}|u(t)|^{2}\,dx  + \beta_{n}\|u(t)\|_{H^{1}}^{2} + \langle \phi_{n}(t),G(u(t),s)\rangle_{L^{2}}
\end{align}
for $t\in[t_{0},t_{1}]$ and $n\ge 1$. Since $\varrho(V_{\infty}-a_{\infty})>\lambda$, we can choose $\ve>0$ and $n_{0}\ge 1$ such that, for all $n\geq n_0$,
\begin{align*}
\phi_{n}(x)(V_{\infty}(x) - a_{\infty}(x)- \lambda) \ge \ve \phi_{n}(x)\quad  \text{for a.a. } x\in\R^{N}.
\end{align*}
By the condition \eqref{ineq-11ab} we obtain
\begin{equation}
\begin{aligned}\label{ineq-j-3b}
\langle\phi_{n}u(t),G(u(t),s)\rangle_{L^{2}} \leq \int_{\R^{N}}\phi_{n} a_{\infty}|u(t)|^{2}\,dx + \int_{\R^{N}}\phi_{n}a_{0} |u(t)|^{2}\,dx  +  \int_{\R^{N}} \phi_{n}b|u(t)|\,dx.
\end{aligned}
\end{equation}
On the other hand, by the H\"older inequality and Sobolev embedding, we have
\begin{equation}
\begin{aligned}\label{eq-r2b}
&\int_{\R^{N}}\phi_{n}a_{0}|u(t)|^{2}\,dx \leq \|\phi_{n}a_{0}\|_{L^{p}}\|u(t)\|_{L^{2p/(p-1)}}^{2} \leq \|\phi_{n}a_{0}\|_{L^{p}}\|u(t)\|_{H^{1}}^{2}.
\end{aligned}
\end{equation}
Combining \eqref{eq-p-1kl}, \eqref{ineq-j-3b} and \eqref{eq-r2b} we obtain
\begin{align*}
 \frac{1}{2}\frac{d}{dt}\int_{\R^{N}}\phi_{n}|u(t)|^{2}\,dx 
 & \leq -\int_{\R^{N}}(V_{\infty}-a_{\infty}-\lambda)\phi_{n}|u(t)|^{2}\,dx 
+ (\beta_{n} + \|\phi_{n}a_{0}\|_{L^{p}})\|u(t)\|_{H^{1}}^{2} + \|\phi_{n}b\|_{L^{2}}\|u(t)\|_{H^{1}} \\
& \leq -\ve\int_{\R^{N}}\phi_{n}|u(t)|^{2}\,dx + (\beta_{n} + \|\phi_{n}a_{0}\|_{L^{p}})\|u(t)\|_{H^{1}}^{2} + \|\phi_{n}b\|_{L^{2}}\|u(t)\|_{H^{1}} \\
& \leq -\ve\int_{\R^{N}}\phi_{n}|u(t)|^{2}\,dx + \ve\gamma_{n}
\end{align*}
for $t\in[t_{0},t_{1}]$ and $n\ge n_0$, where we define $\gamma_{n}:=\ve^{-1}R^{2}(\beta_{n} + \|\phi_{n}a_{0}\|_{L^{p}}) + R\|\phi_{n}b\|_{L^{2}}$. Since $a_{0}\in L^{p}(\R^{N})$ and $b\in L^2(\R^n)$, we have $\gamma_{n}\to 0^+$ as $n\to\infty$. Integrating the above inequality, for any $n\ge n_{0}$, gives
\begin{align*}
\int_{\R^{N}}\phi_{n}|u(t)|^{2}\,dx & \leq e^{-2\ve(t-t_{0})}\int_{\R^{N}}\phi_{n}|u(t_{0})|^{2}\,dx + \gamma_{n}(1-e^{-2\ve(t-t_{0})}) \\
& \leq e^{-2\ve(t-t_{0})}\int_{\R^{N}}\phi_{n}|u(t_{0})|^{2}\,dx + \gamma_{n},
\end{align*}
which completes the proof. 
\end{proof}
As an immediate consequence we obtain the following admissibility result.
\begin{corollary} \label{cor-admissibility}
Let the assumptions of Proposition \ref{prop-tail} hold. If $\rho(V_{\infty}-a_{\infty})>\lambda$, then an arbitrary bounded set $M\subset H^1(\R^N)$ is admissible with respect to the family of semiflows $\{\Psi^{(s)}\}_{ s\in[0,1]}$.
\end{corollary}
\begin{proof}
Let us assume that $(s_{m})$ in $[0,1]$, $(t_{m})$ in $[0,\infty)$ and $(\bar u_{m})$ in $H^{1}(\R^{N})$ are sequences such that $s_{m}\to s_{0}$ and $t_{m}\to \infty$ and $\{\Psi^{(s_{m})}_{t}(\bar u_{m}) \ | \ t\in[0,t_{m}]\} \subset M$ for $m\ge1$. We show that the set $\{\Psi^{(s_{m})}_{t_{m}}(\bar u_{m})\}$ is relatively compact in $H^{1}(\R^{N})$. Let us assume that $R>0$ is such that $M\subset B_{H^{1}}(0,R)$ and let $u_{m}:[0,\infty)\to H^{1}(\R^{N})$ be a solution of the semiflow $\Psi^{(s_{m})}$ starting at $\bar u_{m}$. In view of Proposition \ref{semiflow-continuity-prop} (ii) it is enough to check that the set $\{u_{m}(t_{m}) \ | \ n\ge 1\}$ is relatively compact in $L^{2}(\R^{N})$. To this end we use  Proposition \ref{prop-tail} to obtain $\ve>0$, a sequence $(\gamma_n)$ with $\gamma_n\to 0^+$ as $n\to\infty$ and $n_0\geq 1$, such that
\begin{equation*}
\int_{\R^{N}}\phi_{n}|u_{m}(t_{m})|^{2}\,dx \leq e^{-2\ve t_{m}} \|\bar u_{m}\|_{L^2}^2 + \gamma_n \leq e^{-2\ve t_{m}}R^{2} + \gamma_{n}
\end{equation*}
for all $n\geq n_0$ and $m\ge 1$. This together with the Rellich-Kondrachov theorem (applied for the balls of radius $n$ in $\R^N$) shows that the set $\{ u_{m}(t_{m})\}$ is relatively compact in $L^{2}(\R^{N})$ as desired. 
\end{proof}

\section{The non-resonant case}

In this section we shall prove Theorem \ref{noneresonat-case-theorem}. The following extended versions of the Conley index formulae \cite[Th. 3.3 and 3.6]{Prizzi-FM} will be crucial.
\begin{theorem}\label{noresonant-index-formulae}
Suppose $f:\R^N\times \R\to \R$ satisfies the conditions $(f1)$, $(f2)$, $(f3)$  and $\varrho (V_\infty-a_\infty)>\lambda$.\\
\noindent\makebox[5.5mm][r]{$(i)$} \parbox[t][][t]{168mm}{If there there exists a potential $\alpha$ satisfying $(PT)$ such that
$$
\lim_{u\to\, 0} \frac{f(x,u)}{u}=\alpha(x) \ \mbox{ for a.e. } \ x\in\R^N,
$$
$\varrho(V_\infty-\alpha_\infty)>\lambda$ 
and $\lambda \not \in \sigma (-\Delta+V-\alpha)$,
then $K_0:=\{0\}$ is an isolated invariant set with respect to the semiflow $\Phi$ and 
$$
h(\Phi, K_0)= \Sigma^{d^{-}(V-\alpha,\,\lambda)}.
$$}\\[5pt]
\noindent\makebox[5.5mm][r]{$(ii)$} \parbox[t][][t]{168mm}{If there exists a potential $\omega$ satisfying $(PT)$ such that
$$  
\lim_{|u|\to \infty} \frac{f(x,u)}{u}=\omega (x) \ \mbox{ for a.e. } \ x\in\R^N, 
$$
$\varrho(V_\infty-\omega_\infty)>\lambda$ and
$\lambda \not \in \sigma (-\Delta+V-\omega)$,
then the set $K_\infty (\Phi)$, consisting of all $\bar u\in H^1(\R^N)$ such that $u(0)=\bar u$ for some bounded solution $u:\R\to H^1(\R^N)$ of $\Phi$, is a compact isolated invariant set with respect to $\Phi$ and 
$$
h(\Phi, K_\infty (\Phi)) = \Sigma^{d^{-}(V-\omega,\,\lambda)}.
$$}
\end{theorem}
\begin{proof}
(i) Consider the map $G:H^1(\R^N)\times [0,1] \to L^2(\R^N)$ given, for any $u \in H^1(\R^N)$ and $s\in [0,1]$, by
$$
[G(u,s)](x):= (1-s) F(u)(x) + s \alpha(x) u(x) \quad \text{for a.e. } x\in\R^N,
$$
where $F$ is given by \eqref{Nemytzki_f}. It can be easily seen that $G$ satisfies conditions $(G1)$ and $(G2)$ from Section 4. Let $\{\Psi^{(s)}\}_{s\in [0,1]}$, be the family of semiflows generated by the equation \eqref{eq-diff-g}. We claim that there exists $r_0>0$ such that the set
$N_r :=\{ u\in H^1(\R^N) \ \mid \ \|u\|_{H^1} \leq r\}$ has the property 
\begin{equation}\label{eq-inv-set}
\mathrm{Inv}_{\Psi^{(s)}}\,N_{r} = \{0\} \ \text{ for all } r\in (0,r_0] \ \text{ and } s\in [0,1].
\end{equation}
Indeed, if we suppose to the contrary, then we get a sequence of $(s_m)$ in $[0,1]$ such that, for any $m\geq 1$, there exists a bounded nontrivial full solution $u_m$ of the semiflow $\Psi^{(s_m)}$ such that $$\|u_m(0)\|_{H^1}\geq (1-1/2m)\cdot\sup_{t\in\R} \|u_m(t)\|_{H^1} > 0 \quad \text{and}\quad \rho_m:=\|u_m(0)\|_{H^1}\to 0^+\quad \text{as} \ \  m\to\infty.$$
Note that, for each $m\geq 1$, $v_m:=\rho_m^{-1} u_m$ solves the equation
$$
\dot v(t) = -A v(t) + F_m (v(t)), \quad t\in\R,
$$
with $F_m:H^1(\R^N) \to L^2(\R^N)$ given, for any $u\in H^1(\R^N)$, by
$$
[F_m(u)](x) := f_m(x,u(x))\quad  \text{for a.e. } x\in\R^N
$$
where $f_m:\R^N\times \R\to\R$ is defined by
$$
f_m (x,u):=(1-s_m) \rho_m^{-1} f(x,\rho_m u) + s_m \alpha(x)u \ \mbox{ for a.e. }\ x\in\R^N \mbox{ and } u\in\R.
$$
Clearly
\begin{equation}\label{v_m-is-H1-bounded}
\|v_m(0)\|_{H^1}= 1   \  \ \text{ and  } \ \ \sup_{t\in\R} \|v_m(t)\|_{H^1}\leq (1-1/2m)^{-1} \leq 2 \  \text{ for all } \ m\geq 1.
\end{equation}
On the other hand, by the assumption $(f3)$, for all $m\geq 1$, we have 
\begin{align*}
uf_m (x,u)& =(1-s_m) u\rho_m^{-1} f(x,\rho_m u) + s_m \alpha(x)|u|^{2} \\
& \leq (1-s_m)a(x)|u|^{2} + s_m \alpha(x)|u|^{2} \leq \tilde a (x)|u|^{2},
\end{align*}
where $\tilde a = \tilde a_{\infty}+\tilde a_{0}$ with $\tilde a_{\infty}:=\max\{a_{\infty},\alpha_{\infty}\}$ and $\tilde a_{0}:=a_{0} + \alpha_{0}$. Here $a_0\in L^p(\R^N)$ and $\alpha_0\in L^r(\R^N)$, where the exponents  $p=p(a_0)>2$ and $r=r(\alpha_0)\ge 2$ are restricted by the conditions $(KR)$ and $(PT)$, respectively. 
Since $\varrho (V_\infty -a_\infty)>\lambda$ and $\varrho (V_\infty-\alpha_\infty)>\lambda$ we infer that 
$$
\varrho \left(V_\infty - \tilde a_\infty\right)>\lambda.
$$
Hence, in view of Proposition \ref{prop-tail} with $G(u,s)=F_m(u)$ for $u\in H^1(\R^N)$ and $s\in [0,1]$, we see that there exist $\varepsilon>0$, sequence $(\gamma_n)$ with $\gamma_n\to 0^+$ and $n_0\geq 1$ (independent of $m$) such that, for all $m\geq 1$ and $n\geq n_0$, we have
$$
\int_{\R^N\setminus B(0,n)} |v_m (t)|^2\,d x \leq e^{-2\varepsilon (t-r)}\|v_{m}(r)\|^{2}  + \gamma_n \leq 4e^{-2\varepsilon (t-r)}+\gamma_n,\quad r<t,
$$
where in the last inequality we used the estimate \eqref{v_m-is-H1-bounded}. Fixing $t\in\R$ and passing with $r \to -\infty$, we get
$$
\int_{\R^N\setminus B(0,n)} |v_m (t)|^2\,d x \leq \gamma_n,\quad n\ge n_{0}, \ m\ge 1.
$$
Since the sequence $(v_m(t))$ is bounded in $H^1(\R^N)$, by the Rellich-Kondrachov compact embedding theorem, we infer that the set of restricted functions $\{v_m(t)|_{B(0,n)} \ | \ m\ge 1\}$ is relatively compact in $L^2(B(0,n))$ for any $n\ge 1$. Combining this with the fact that $\gamma_n\to 0^{+}$ we see that, for any $t\in\R$, the set $\{ v_m(t)\}_{m\geq 1}$ is relatively compact in $L^2(\R^N)$. Take any $\tau>0$ and let $\big( v_{m_k}(t-\tau)\big)$ be a subsequence of $(v_m(t-\tau))$ such that $v_{m_k}(t-\tau)\to \bar v$ in $L^2(\R^N)$. Using the weak compactness argument we can suppose that $\bar v\in H^1(\R^N)$. Since, for all $u\in\R$ and a.e. $x\in\R^N$,
$$
f_m (x,u) \to \alpha(x)u, \quad m\to \infty,
$$
we can use \cite[Proposition 4.3]{Cw-Luk-2021} taking into consideration \eqref{v_m-is-H1-bounded}, to see that
$v_{m_k}$ converges in $H^1(\R^N)$, uniformly on compact subsets of $(t-\tau,\infty)$, to the solution $v:[t-\tau, \infty)\to H^1(\R^N)$ of
\begin{equation}\label{eq-a-alpha}
\dot v(t)=-A v(t)+ m_\alpha v(t), \quad t>0,
\end{equation}
where $m_\alpha:H^1(\R^N)\to L^2(\R^N)$ is given by $m_\alpha(u)(x):=\alpha(x)u(x)$ for a.e. $x\in\R^N$ and all $u\in H^1(\R^N)$. Since $\tau>0$ is arbitrary, we can combine this with a diagonal type argument to obtain the existence of a bounded solution $v:\R\to H^1(\R^N)$ of the equation \eqref{eq-a-alpha} such that 
$\|v(0)\|_{H^1}= 1$. If we consider the functional $\mathcal{E}_\alpha:H^1(\R^N)\to \R$, given by 
$$
\mathcal{E}_\alpha (u):=\frac{1}{2}\int_{\R^N} \big(|\nabla u|^2 +(V(x)-\alpha(x)-\lambda)|u|^2\big)\, \,d x,\quad u\in H^{1}(\R^{N}),
$$
then, by Proposition \ref{prop-fradc-d-c}, we have
\begin{align*}
\frac{\d}{\d t} \mathcal{E}_\alpha (v(t)) =-\int_{\R^N} |\dot v(t)|^{2} \,d x \ \text{ for all } \ t\in\R, 
\end{align*}
i.e. the semiflow generated by \eqref{eq-a-alpha} is gradient-like with respect to the Lyapunov functional $\mathcal{E}_\alpha$. Hence, since $v$ is nontrivial and bounded, there exists a nontrivial equilibrium point in the set $\alpha(u)\cup\omega(u)$. Then we have $\Ker (-A+m_\alpha)=\Ker(-\Delta+V-\alpha-\lambda)\neq \{0\}$, which is a contradiction proving the existence of $r_0>0$ such that
\eqref{eq-inv-set} holds.\\
\indent Consequently from \eqref{eq-inv-set} it follows that $\mathrm{Inv}_{\Psi^{(s)}}\,N_{r} = \{0\}\subset \mathrm{int}\,N_{r}$ for $s\in[0,1]$ and $r\in(0,r_{0}]$ and furthermore, in view of Proposition \ref{semiflow-continuity-prop} and Corollary \ref{ineq-11ab}, the set $N_r$ is admissible with respect to the family of semiflows $\{\Psi^{(s)}\}_{s\in [0,1]}$. Hence, we can  apply the homotopy invariance property (H4) to get
\begin{equation}\label{eq-h-ind}
h(\Phi, K_0) = h(\Psi^{(0)}, K_0) = h(\Psi^{(1)}, K_0).
\end{equation}
Since $\Psi^{(1)}$ is the semiflow generated by \eqref{eq-a-alpha} and
$\Ker (-A+m_\alpha)=\{0\}$, we infer from Theorem \ref{linear-conley-index}, that
$$
h(\Psi^{(1)}, K_0) = \Sigma^{d^{-}(V-\alpha,\,\lambda)},
$$
which together with \eqref{eq-h-ind} ends the proof of the assertion (i).\\
\indent (ii) In this case we consider the family of semiflows $\{\Psi^{(s)}\}_{s\in[0,1]}$ determined by the equation \eqref{eq-diff-g}, where 
the map $G$ is given, for any $u \in H^1(\R^N)$ and $s\in [0,1]$, by the formula
$$
[G(u,s)](x):= (1-s) F(u)(x) + s \omega(x) u(x) \quad \text{for a.e. } x\in\R^N.
$$
We claim that there exists $R>0$ such that if $u$ is a bounded in $H^1(\R^N)$ solution of the semiflow $\Psi^{(s)}$ for some $s\in[0,1]$, then $\|u(t)\|_{H^{1}} < R$ for all $t\in\R$. If we suppose to the contrary, then there is a sequence of $(s_m)$ in $[0,1]$ such that, for any $m\geq 1$, there exists a bounded solution $u_{m}$ of the semiflow $\Psi^{(s_m)}$ such that $\rho_m:=\sup_{t\in\R} \|u_m(t)\|_{H^1}\to \infty$ and $\|u_m(0)\|_{H^1}\geq \rho_m-1$. Taking $v_m := \rho_m^{-1} u_m$ for $m\geq 1$, we can show in a similar manner as in the proof of (i), that $\Ker(-\Delta+V-\omega-\lambda)\neq \{0\}$, which is a contradiction.\\
\indent Therefore, if we take 
$N_R :=\{ u\in H^1(\R^N) \, \mid \, \|u\|_{H^1} \leq R\}$, then we have 
\begin{equation*}
K_\infty (\Phi) \subset N_{R}\quad\text{and}\quad K_{s}:=\mathrm{Inv}_{\Psi^{(s)}}\,N_{R} \subset \mathrm{int}\,N_{R}\ \text{ for all } \ s\in [0,1].
\end{equation*}
Similarly as in the point $(i)$, Proposition \ref{semiflow-continuity-prop} and Corollary \ref{ineq-11ab} show that the set $N_R$ is admissible with respect to the family of semiflows $\{\Psi^{(s)}\}_{s\in [0,1]}$. Hence, we can  apply the homotopy property (H4) to get
\begin{equation}\label{eq-h-ind2}
h(\Phi, K_\infty(\Phi)) = h(\Psi^{(0)}, K_0) = h(\Psi^{(1)}, K_1).
\end{equation}
Since $\Psi^{(1)}$ is the semiflow generated by the equation \eqref{eq-a-alpha} with $\alpha$ replaced by $\omega$ and $\Ker (-A+m_\omega)=\{0\}$, we can again apply Theorem \ref{linear-conley-index} to obtain 
\begin{equation*}
h(\Psi^{(1)}, K_1) = \Sigma^{d^{-}(V-\omega,\,\lambda)}.
\end{equation*}
Combining the above equation with \eqref{eq-h-ind2} completes the proof of the assertion (ii).
\end{proof}

\begin{proof}[Proof of Theorem \ref{noneresonat-case-theorem}]
From Theorem \ref{noresonant-index-formulae} (ii) it follows that $K_{\infty}(\Phi)$ is an isolated invariant set
with respect to $\Phi$ and
$$
h(\Phi, K_\infty(\Phi)) = \Sigma^{d^-(V-\omega,\,\lambda)}.
$$
On the other hand, by Theorem \eqref{noresonant-index-formulae} (i), $K_0=\{0\}$ is also an isolated $\Phi$-invariant set and 
$$
h(\Phi, K_0)=\Sigma^{d^-(V-\alpha,\,\lambda)}.
$$
Since, by assumption, $d^-(V-\alpha,\lambda) \neq d^-(V-\omega,\lambda)$, we infer that 
$$
\overline 0 \neq \Sigma^{d^-(V-\alpha,\,\lambda)} = h(\Phi, K_0 ) \neq h(\Phi, K_\infty(\Phi)) = \Sigma^{d^-(V-\omega,\,\lambda)} \neq \overline{0}.
$$
This together with Theorem \ref{rybakowski-irreducible}, implies the existence of a nontrivial full bounded solution $u:\R \to H^1(\R^N)$ for the semiflow $\Phi$ such that either $\alpha(u)\subset K_0$ or $\omega(u)\subset K_0$. Recall that, in view of Proposition  \ref{prop-fradc-d-c}, $\Phi$ is gradient-like, which implies that
both $\alpha(u)$ and $\omega(u)$ consists of equilibria. Since $u$ is nontrivial and bounded and hence either $\alpha(u)$ or $\omega(u)$ contains a nonzero equilibrium point. Then the zero solution is in the remaining limit set of $u$ and the proof is completed. 
\end{proof}

\section{The resonant case}

Throughout the whole section we assume that at least the assumptions of Theorem \ref{30042019-1204} are satisfied, i.e.
$f$ satisfies $(f1)'$ and $(f2)$, $V$ is a Kato-Rellich type potential, $\varrho (V_\infty)>\lambda$ and $\lambda\in \sigma(-\Delta+V)$. Let us define $G:H^1(\R^N)\times [0,1]\to L^2(\R^N)$ by
\begin{equation}\label{def-g-r}
G(u, s): = P F (P u+s Q u) + s Q F(P u+sQu) \quad \text{for} \ \ u\in H^1(\R^N), \ s\in [0,1],
\end{equation}
where $P$ and $Q=Q_{+}+Q_{-}$ are as in Section 3. 
\begin{lemma}
The mapping $G$ satisfies conditions $(G1)$, $(G2)$ and there exists $b\in L^2(\R^N)$ such that, for all $u\in H^1(\R^N)$,
\begin{equation}\label{b-bound-for-G}
|G(u,s)(x)|\leq b(x) \ \text{ for a.e. } x\in\R^N.
\end{equation}
\end{lemma}
\begin{proof}
Indeed, for $(G1)$ note that, by \eqref{eq-lip}, there is $L>0$ such that, for all $u,v\in H^1(\R^N)$ and $s\in [0,1]$,
\begin{align*}
\|G(u,s)-G(v,s)\|_{L^{2}} & \leq 2\|F (P u+s Q u) - F (P v+s Q v)\|_{L^{2}} \\
& \leq 2L\|P (u-v)+s Q (u-v)\|_{L^{2}}
 \leq 4L\|u-v\|_{L^{2}}.
\end{align*}
Furthermore, by $(f1)'$, we have
\begin{align}\label{eq-s-g-1}
\|G(u,s)\|_{L^{2}} &  \leq 
2\|F (P u+s Q u)\|_{L^{2}} \leq 2\|m\|_{L^{2}},\quad u\in H^{1}(\R^{N}), \ s\in [0,1],
\end{align}
which implies $(G2)$.\\ 
\indent In order to verify the property \eqref{b-bound-for-G} observe that, given $u\in H^{1}(\R^{N})$ and $s\in [0,1]$, we have
\begin{equation}\label{eq-11bb}
G(u, s) = s F(P u+sQu) + (1-s) P F (P u+s Q u) \quad \text{ almost everywhere on }\  \R^{N}.
\end{equation}
If we take an orthogonal basis of $X_0$ consisting of elements 
$\varphi_k$ for $k=1,\dots, d$, where $d=\dim X_0$, then
\begin{align}\label{def-q-p}
P F (P u+s Q u) = \sum_{k=1}^{d} e_{k}(u,s)\varphi_{k}, \quad u\in H^{1}(\R^{N}), \ s\in [0,1],
\end{align}
where 
\begin{align*}
e_{k}(u,s) := \int_{\R^{N}} F (P u+s Q u)\varphi_{k}\,dx,\quad 1\leq k\leq d.
\end{align*}
Clearly, by $(f1)'$, for all $u\in H^1(\R^N)$ and $s\in [0,1]$, 
\begin{align*}
|e_{k}(u,s)| \leq \|F (P u+s Q u)\|_{L^{2}} \|\varphi_{k}\|_{L^{2}} \leq \|m\|_{L^{2}} \|\varphi_{k}\|_{L^{2}}.
\end{align*}
and consequently, for a.e. $x\in\R^N$,
\begin{align}\label{eq-pp1}
|P F (P u+s Q u)(x)| \leq \sum_{k=1}^{d} |e_{k}(u,s)||\varphi_{k}(x)|\leq \|m\|_{L^{2}}\sum_{k=1}^{d} \|\varphi_{k}\|_{L^{2}} |\varphi_{k}(x)|.
\end{align}
Combining \eqref{eq-11bb} and $(f1)'$ gives, for all $u\in H^1(\R^N)$ and a.e. $x\in\R^N$,
\begin{align*}
|G(u,s)(x)| \leq |m(x)| + |P F (P u+s Q u)(x)|.
\end{align*}
and hence, by the inequality \eqref{eq-pp1}, the condition \eqref{b-bound-for-G} holds with $b\in L^{2}(\R^{N})$ given by 
$b(x):=|m(x)| + \|m\|_{L^{2}} \sum_{k=1}^{d} \|\varphi_{k}\|_{L^{2}} |\varphi_{k}(x)|$ for a.e. $x\in\R^{N}$. 
\end{proof}

Let us consider the family of semiflows $\{\Psi^{(s)}\}_{s\in [0,1]}$, generated by the equation \eqref{eq-diff-g} with $G$ given by the formula \eqref{def-g-r}. Since $G(u,1)=F(u)$ for $u\in H^1(\R^N)$, we see that $\Psi^{(1)}=\Phi$ where $\Phi$ is the semiflow defined by $(P)_\lambda$. Since $G$ satisfies \eqref{b-bound-for-G}, the conditions \eqref{ineq-11ab} hold and we can apply Corollary \ref{cor-admissibility} to see that any bounded subset of $H^1(\R^N)$ is admissible with respect to the family of semiflows $\Psi=\{\Psi^{(s)}\}_{s\in [0,1]}$.\\
\indent Below we obtain further compactness and geometric properties of the family $\Psi$.

\begin{proposition} \label{Q-boundedness}
Let $K_{\infty} (\Psi)$ be the set of all $\bar u\in H^1(\R^N)$ such that there exist $s\in[0,1]$ and a bounded solution $u:\R\to H^1(\R^N)$ of $\Psi^{(s)}$ with $u(0)=\bar u$. Then the set $Q K_{\infty} (\Psi)$ {\em(}the projection of $K_\infty (\Psi)$ onto $X_- \oplus X_+${\em)} is bounded in $H^1(\R^N)$ and relatively compact in $L^2(\R^N)$.
\end{proposition}
\begin{proof}
Let $u:\R \to H^1(\R^N)$ be a bounded solution of $\Psi^{(s)}$ for some $s\in [0,1]$. Then $u$ is also a mild solution of the equation \eqref{eq-diff-g} and we have the Duhamel formula
\begin{equation}\label{duhamel-eq}
u(t) = S(t-t') u(t')+ \int_{t'}^{t} S(t-\tau) G(u(\tau),s)\,d\tau, \quad t'<t.
\end{equation}
By applying the projection operator $Q_+$ we get
$$
Q_+ u(t) = S(t-t')Q_+ u(t') + \int_{t'}^{t} S(t-\tau) Q_+ G(u(\tau),s) \, \d\tau.
$$
In view of \eqref{X-plus-semigroup-ineq} and the boundedness of $u$, by passing to the limit with $t'\to -\infty$, we get
$$
Q_+ u(t) = \int_{-\infty}^{t} S(t-\tau) Q_+ G(u(\tau),s) \, \d\tau, \quad t\in\R,
$$
where the integral is convergent in $X=L^{2}(\R^{N})$. Therefore, by \eqref{X-plus-semigroup-ineq} and \eqref{eq-s-g-1}, we have
\begin{equation}
\begin{aligned}\label{ineq-11ac}
\|Q_+ u(t) \|_{H^1} & \leq \int_{-\infty}^{t} K (t-\tau)^{-1/2} e^{-\rho_+(t-\tau)} \|Q_{+}G(u(\tau),s)\|_{L^2} \, d\tau \\
& \leq 2 K \|m\|_{L^{2}} \int_{0}^{+\infty} z^{-1/2} e^{-\rho_+ z}\, \d z
\end{aligned}
\end{equation}
for any $t\in\R$. Let us recall that the restriction of $\{S(t)\}_{t\geq 0}$ to the space $X_{-}$ can be extended to a $C_0$-group $\{ S(t)\}_{t\in\R}$ on this space. Hence, applying with the projection $Q_-$ on both sides of the equation \eqref{duhamel-eq} and using the $C_0$-group property, we get
$$
Q_- u(t') = S(t'-t) Q_- u(t) - \int_{t'}^{t} S(t'-\tau) Q_- G(u(\tau), s)\, d\tau, \quad t'<t.
$$
Using \eqref{X-minus-group-ineq} and passing to the limit with $t\to +\infty$ we obtain
\begin{equation}\label{eq-q-min}
Q_- u(t') = - \int_{t'}^{+\infty} S(t'-\tau)Q_- G(u(\tau), s)\, d\tau, \quad t'\in\R,
\end{equation}
where the integral is convergent in $X=L^{2}(\R^{N})$. Applying \eqref{X-minus-group-ineq} and \eqref{eq-s-g-1} to the equation \eqref{eq-q-min} yields
$$
\|Q_- u(t')\|_{L^2} \leq \int_{t'}^{+\infty} e^{(t'-\tau)\rho_-}\|Q_{+}G(u(\tau),s)\|_{L^2}\,d\tau\leq 
2\|m\|_{L^{2}}\int_{t'}^{+\infty} e^{(t'-\tau)\rho_-}\,d\tau \leq 2\|m\|_{L^{2}}/\rho_-.
$$
This combined with \eqref{ineq-11ac} implies the existence of $R_\infty>0$ such that
\begin{equation}\label{ineq-b-1}
\sup_{t\in\R} \| Q u(t) \|_{H^1} < R_\infty, \quad t\in\R
\end{equation}
for any full bounded solution $u$ of $\Psi^{(s)}$ with $s\in [0,1]$.\\
\indent To show the relative compactness, we act on \eqref{eq-diff-g} with the operator $Q$ and we write $w(t) := Qu(t)$ for $t\in\R$, to obtain the equation
\begin{align*}
\dot w(t) = -A w(t) + QG(u(t),s),\quad t>0.
\end{align*}
By Proposition \ref{classic_tail_estimates}, there exist a sequence $(\beta_{n})$ with $\beta_{n}\to 0^+$ as $n\to\infty$ and $n_0\geq 1$, such that if $u$ is a full solution of the semiflow $\Psi^{(s)}$, where $s\in[0,1]$, then
\begin{align}\label{eq-p-1}
\frac{1}{2}\frac{d}{dt}\langle\phi_{n}w(t), w(t)\rangle_{L^{2}} \leq -\int_{\R^{N}}(V_{\infty}-\lambda)\phi_{n}|w(t)|^{2}\,dx  + \beta_{n}\|w(t)\|_{H^{1}}^{2} + \langle w(t),\phi_{n}QG(u(t),s)\rangle_{L^{2}}
\end{align}
for $t\in\R$. From the assumption $\varrho(V_{\infty})>\lambda$, there are $\ve>0$ and $n_{0}\ge 1$ such that, for all $n\geq n_0$,
\begin{align}\label{ineq-12b}
\phi_{n}(x)(V_{\infty}(x)- \lambda) \ge \ve \phi_{n}(x)\quad  \text{for a.a. } x\in\R^{N}.
\end{align}
By \eqref{def-g-r} and $(f1)'$, we have
\begin{align}\label{eq-a-1}
\|\phi_{n}QG(u(t),s)\|_{L^{2}} = \|\phi_{n}QF(Pu(t) + sQu(t))\|_{L^{2}}  
\leq \|\phi_{n}m\|_{L^{2}}  + \|\phi_{n}PF(Pu(t) + sQu(t))\|_{L^{2}}.
\end{align}
Since $m\in L^{2}(\R^{N})$, we have $\delta^{1}_{n} := \|\phi_{n}m\|_{L^{2}} \to 0$ as $n\to \infty$. On the other hand, the image of the map $F$ is contained in the ball $B_{L^2}(0,\|m\|_{L^2})$
and hence its projection  onto the finite dimensional space $X_{0}$ is relatively compact in $L^2(\R^N)$. Therefore, by the Kolmogorov-Riesz theorem (see \cite{Hanche}), there is a sequence $(\delta^{2}_{n})$, depending on $\|m\|_{L^2}$, such that $\delta^{2}_{n}\to 0^+$ as $n\to \infty$ and $\|\phi_{n}PF(Pu(t) + sQu(t))\|_{L^{2}} \leq \delta^{2}_{n}$ for $t\in\R$ and $s\in[0,1]$. Consequently, we get $\delta_{n} := \delta^{1}_{n}+\delta^{2}_{n}$ that is independent of $s$, $u$ and $t$ such that
\begin{align*}
\langle w(t),\phi_{n}QG(u(t),s)\rangle_{L^{2}} \leq \delta_{n}\|w(t)\|_{H^{1}}, \quad t\in\R,
\end{align*}
and $\delta_{n}\to 0^+$ as $n\to \infty$. Now we combine this with \eqref{eq-p-1} and \eqref{ineq-12b} and use \eqref{ineq-b-1} to get
\begin{align*}
\frac{1}{2}\frac{d}{dt}\langle\phi_{n}w(t), w(t)\rangle_{L^{2}} & \leq -\ve\int_{\R^{N}}\phi_{n}|w(t)|^{2}\,dx  + \beta_{n}R_{\infty}^2 +  \delta_{n} R_\infty,\quad t\in\R.
\end{align*}
By integrating the above inequality we obtain, for all $n\geq n_0$ and $t,r\in\R$ with $t>r$,
\begin{align*}
\int_{\R^{N}}\phi_{n}|w(t)|^{2}\,dx & \leq e^{-2\ve (t-r)}\int_{\R^{N}}\phi_{n}|w(r)|^{2}\,dx + \frac{1}{2\ve}(1 - e^{-2\ve(t-r)})(\beta_{n}R_{\infty}^2 +  \delta_{n} R_\infty)\\
& \leq e^{-2\ve (t-r)}R^2 + \frac{1}{2\ve}(\beta_{n}R_{\infty}^2 +  \delta_{n} R_\infty).
\end{align*}
Passing with $r\to-\infty$ we see that, for all $t\in\R$ and $n\geq n_0$,
$$
\int_{\R^{N}}\phi_{n}|w(t)|^{2}\,dx \leq \frac{1}{2\ve}(\beta_{n}R_{\infty}^2 +  \delta_{n} R_{\infty}),
$$
which by the Rellich-Kondrachov theorem shows that 
the set $QK_\infty(\Psi)$ is relatively compact in $L^2(\R^N)$. Thus the proof of the proposition is completed. 
\end{proof}

From now on, we assume that $f$ satisfies additionally one of the resonance conditions $(LL)_{\pm}$ or $(SR)_{\pm}$. Both Landesman-Lazer and sign conditions have geometric implications in the phase space of the parabolic semiflow $\Phi$.
\begin{proposition}\label{geometric-cond} {\em (see \cite[Lemma 5.2]{Cw-Kr-2019})}
Let the map $F$ be given by \eqref{Nemytzki_f} and $M\subset H^1(\R^N)$. If either \\
\indent {\em (i) $(LL)_\pm$} holds and $M$ is bounded in $L^2(\R^N)$; or\\
\indent {\em (ii)} $(SR)_\pm$ holds and $M$ is relatively compact in $L^2(\R^N)$,\\
then there exists  $R_0>0$ and $\alpha>0$ such that, for all $v\in X_0$ with $\|v\|_{L^2}\geq R_0$ and $w\in M$
\begin{equation*}
\pm \langle v, F (v+w)\rangle_{L^2} > \alpha.
\end{equation*}
\end{proposition}
By Proposition \ref{Q-boundedness}, there exists $R_{\infty}>0$ such that for any bounded solution $u$ of $\Psi^{(s)}$ with $s\in [0,1]$,
\begin{equation}\label{Q-bound}
Qu(t)\in B_{H^1} (0,R_\infty), \quad \text{ for all }  \ t\in\R,
\end{equation}
and the set $QK_\infty(\Psi)$ is relatively compact in $L^2(\R^N)$. Therefore, by the use of Proposition \ref{geometric-cond}, we can choose $R_0>0$ such that 
\begin{align}\label{geom-cond2}
\pm \langle v, F (v+sw)\rangle_{L^2} > \alpha\quad \text{for \ $s\in[0,1]$, $w\in QK_\infty(\Psi)$ \ and \ $v\in X_0$ \ with \ $\|v\|_{L^2}\geq R_0$}
\end{align}
(the sign in \eqref{geom-cond2} is the same as in the Landesman-Lazer or sign condition that is satisfied by $f$). 
Now let us define
\begin{equation}\label{nbhd-def}
N:= \{ u\in H^1(\R^N) \mid \|Qu\|_{H^1} \leq R_\infty \ \text{ and } \ \|P u\|_{L^2}\leq R_0 \}.
\end{equation}
The set $N$ appears to be an isolating neighborhood with respect to the family $\Psi$ and plays a key rule in the Conley index formula derivation that is the main result of the section.

\begin{theorem}\label{18072019-0920}
Let $\Phi$ be the semiflow generated by \eqref{abs-eq} and $K_\infty (\Phi)$ be the set of all $\bar u\in H^1(\R^N)$, for which there exists a bounded solution $u:\R\to H^1(\R^N)$ of  $\Phi$ such that $u(0)=\bar u$.  Then $K_\infty(\Phi)$ is an isolated invariant set such that 
\begin{equation}\label{eq-k-inf}
K_\infty (\Phi) = \mathrm{Inv}_{\Phi}(N) \subset \mathrm{int}\, N
\end{equation}
where $N$ is given by \eqref{nbhd-def} and its Conley index with respect to $\Phi$ is given by the formula
\begin{equation}\label{eq-m-1}
h( \Phi , K_\infty (\Phi) ) =\left\{\begin{aligned}
&\Sigma^{d} && \text{ if  $(LL)_+$  or $(SR)_+$ holds},\\
&\Sigma^{d^{-}} && \text{ if  $(LL)_-$  or $(SR)_-$ holds},
\end{aligned}\right.
\end{equation}
where $d^{-}:= d^{-}(V,\lambda)$ and $d:=d^{-} + \dim \Ker (-\Delta+V-\lambda)$.
\end{theorem}
\begin{proof}
Let $u:\R\to H^1(\R^N)$ be an arbitrary bounded solution of  of $\Psi^{(s)}$ for some $s\in[0,1]$. We show that $u(R)\subset\mathrm{int}\,N$. Indeed, suppose to the contrary, that, for some $t_0\in \R$, either $\|Qu(t_0)\|_{L^2} \geq R_\infty$ or $\|Pu(t_0)\|_{L^2} \geq R_0$. From \eqref{Q-bound} it follows that $\|Pu(t_0)\|_{L^2}=R_0$  and, by the geometric condition \eqref{geom-cond2}
\begin{equation*}
\pm \left.\frac{d}{dt}\| Pu(t) \|_{L^2}^2 \right|_{t=t_0}  = \pm 2 \langle Pu(t_0),  PG(u(t_0),s)\rangle_{L^2}
= \pm 2 \langle Pu(t_0),  F(Pu(t_0)+sQu(t_0))\rangle_{L^2} >2\alpha.
\end{equation*}
This in turn implies that either $\|Pu(t)\|_{L^2}^{2} \geq R_0^{2} + 2\alpha(t-t_0)$ for $t>t_0$ (or $\|Pu(t)\|_{L^2}^{2} \geq R_0^{2} + 2\alpha(t_0-t)$ for $t<t_0$, respectively), which means that $u$ would not be bounded. This gives a contradiction and proves the desired inclusion. Consequently \eqref{eq-k-inf} holds. Furthermore, by Corollary \ref{cor-admissibility}, $N$ is admissible with respect to $\{ \Psi^{(s)}\}_{s\in [0,1]}$  and, by the homotopy invariance of Conley index -- see (H4) in Section 2, we get
$$
h(\Phi, K_\infty (\Phi)) = h(\Psi^{(1)}, K_1) = h(\Psi^{(0)},K_{0}).
$$
Now we observe that $\Psi^{(0)}$ is conjugate to the semiflow $\Phi_P \times \Phi_Q$,
where $\Phi_P:X_0\to X_0$ is generated by $\dot v = P F(v)$, whereas the semiflow $\Phi_Q:  Y \to Y$, where $Y=H^{1}(\R^{N})\cap (X_{+}\oplus X_{-})$ is the linear space defined in Theorem \ref{linear-conley-index}, is generated by the linear equation $\dot w = -A_{|Y} w$ for $t>0$. In particular 
\begin{equation}\label{eq-con-1}
h(\Psi^{(0)}, K_0)= 
h(\Phi_P \times \Phi_Q, \mathrm{Inv}_{\Phi_P}\, (N_P)\times 
\mathrm{Inv}_{\Phi_Q}\, (N_Q)),
\end{equation}
where $N_P :=N\cap X_0 = \{ u \in X_0 \mid \|u\|_{L^2}\leq R_0\}$ and $N_Q :=N\cap Y = \{ u \in Y \mid \|u\|_{H^1}\leq R_\infty\}$.

By use of Theorem \ref{linear-conley-index}, we have $\mathrm{Inv}_{\Phi_Q}\,(N_Q) = \{0\}$
and 
\begin{equation}\label{eq-c-q}
h(\Phi_Q, \{ 0\})= \Sigma^{d_{-}}.
\end{equation}
Clearly, if $u_P:[\delta_1,\delta_2)\to X_0$, where $\delta_1>0$ and $\delta_2\ge 0$, is a solution of the flow $\Phi_P$ and $u_P(0)\in\partial N_P$, then
$\|u_P(0)\|_{L^2}=R_0$ and, by the use of \eqref{geom-cond2}, we have
\begin{equation*}
\pm \left.\frac{d}{dt}\| u_P(t) \|_{L^2}^2\right|_{t=0}  = \pm 2 \langle u_P(0),  PF(u_P(0))\rangle_{L^2}
= \pm 2 \langle u_P(0),  F(Pu_P(0))\rangle_{L^2} >2\alpha.
\end{equation*}
This implies that the set $N_P$ is an isolating block for $\Phi_{P}$ with the exist set 
\begin{equation}
N_P^-=\left\{\begin{aligned}
& \{ u \in X_0 \mid \|u\|_{L^2}= R_1\} &&  \text{ if  $(LL)_+$  or $(SR)_+$ holds},\\
& \emptyset && \text{ if  $(LL)_-$  or $(SR)_-$ holds}.
\end{aligned}\right.
\end{equation}
Therefore we have
\begin{equation}\label{eq-c-p}
h(\Phi_P, \mathrm{Inv}_{\Phi_P}\, (N_P)) = \left\{ \begin{aligned}
& \Sigma^{\dim X_0} && \text{ if  $(LL)_+$ or $(SR)_+$ holds},\\
& \Sigma^{0} && \text{ if  $(LL)_-$ or $(SR)_-$ holds}.
\end{aligned}\right.
\end{equation}
By \eqref{eq-con-1} and \eqref{eq-c-q} and the multiplication property of the Conley index -- see (H3) in Section 2, we have
\begin{align*}
h(\Psi^{(0)},K_{0}) & = h(\Phi_P \times \Phi_Q, \mathrm{Inv}_{\Phi_P}\, (N_P) \times \{0\}) \\
& = h(\Phi_P, \mathrm{Inv}_{\Phi_P}\, (N_P) )\wedge h( \Phi_Q,\{0\}) \\
& = h(\Phi_P, \mathrm{Inv}_{\Phi_P}\, (N_P) )\wedge \Sigma^{d^-},    
\end{align*}
which together with \eqref{eq-c-p} yields the required index formula \eqref{eq-m-1}.
\end{proof}

\begin{proof}[Proof of Theorem \ref{30042019-1204}]
Note that in view of Theorem \ref{18072019-0920} we have $h(\Phi,K_\infty(\Phi))=\Sigma^{k_\infty}$, where 
\begin{equation}\label{def-k-infty}
k_\infty =  \left\{\begin{aligned}
&d^{-}(V,\lambda) + \dim\, X_0 &&  \text{ if  $(LL)_+$ or $(SR)_+$ holds},\\
&d^{-}(V,\lambda) && \text{ if  $(LL)_-$ or $(SR)_-$ holds}.
\end{aligned}\right.
\end{equation}
This implies that the Conley index of
$K_\infty (\Phi)$ is nontrivial and, in view of the existence property of Conley index -- see (H1), we infer that $K_\infty (\Phi)\neq \emptyset$. Hence there exists a full bounded solution for $\Phi$, which allows us to apply the argument with Lyapunov functional \eqref{eq-lyap} from the proof of Theorem \ref{noneresonat-case-theorem}, to get that the sum $\alpha(u)\cup\omega(u)$ contains the desired stationary solution.
\end{proof}

\begin{proof}[Proof of Theorem \ref{30042019-1206}]
By Theorem \ref{noresonant-index-formulae} (i) we see that $K_0 = \{ 0\}$ is an isolated invariant set
with respect to $\Phi$ and $h(\Phi, K_0)=\Sigma^{d^- (V-\alpha,\,\lambda)}$. In view of Theorem \ref{18072019-0920}, $K_\infty(\Phi)$ is also an isolated $\Phi$-invariant set and $h(\Phi,K_\infty(\Phi))=\Sigma^{k_\infty}$, where $k_\infty$ is given by \eqref{def-k-infty}. 
In both cases (i) and (ii) we have
$$
\overline{0}\neq \Sigma^{d^- (V-\alpha,\,\lambda)} = 
h(\Phi, K_0)\neq h(\Phi,K_\infty(\Phi))=\Sigma^{k_\infty}\neq \overline{0}.
$$
Therefore, in virtue of Theorem \ref{rybakowski-irreducible}, there exists a full solution $u:\R\to H^1(\R^N)$ that is nontrivial and either $\alpha(u)\subset K_0$ or $\omega(u)\subset K_0$. 
Using the Lyapunov functional \eqref{eq-lyap}, 
we derive that either $\alpha(u)$ or $\omega(u)$ contains a nonzero equilibrium point and the zero solution is the other side limit set of $u$. Thus the proof of the theorem is completed. 
\end{proof}

\begin{proof}[Proof of Corollary \ref{corollary-resonance}]
If (i) holds, then we have $d^{-}(V-\alpha,\lambda)=
d^{-}(V,\bar\alpha+\lambda)\leq d^{-}(V,\lambda) < k_\infty$, where $k_\infty$ is given by \eqref{def-k-infty}. In case of (ii) we get $d^{-}(V-\alpha,\lambda) = d^{-}(V,\bar\alpha+\lambda)> d^{-}(V,\lambda)=k_\infty$, since $d^{-}(V, \bar\alpha+\lambda)-d^{-}(V,\lambda) \geq \dim\, X_0$  due to the fact that $\bar \alpha >0$.
When (iii) holds, then obviously $k_\infty = d^{-}(V,\lambda) + \dim\, X_0 < d^-(V,\bar \alpha+\lambda) =d^- (V-\alpha,\lambda)$, since, by assumption, there is at least one eigenvalue in the interval  $(\lambda, \bar\alpha +\lambda)$.
Finally, if (iv) holds, then in a similar way, we obtain $k_\infty=d^{-} (V,\lambda) >d^{-}(V, \bar\alpha+\lambda) = d^- (V-\alpha,\lambda)$.
In all four cases we use Theorem \ref{30042019-1206} to derive the assertion.
\end{proof}

\parindent = 0 pt
\end{document}